\documentclass[12pt,reqno]{amsart}

\usepackage[T1]{fontenc}
\usepackage[utf8]{inputenc}
\usepackage{amssymb}
\usepackage{amsmath}
\usepackage{mathptmx}
\usepackage{url}
\usepackage{hyperref}
\usepackage{pgf,tikz}
\usepackage{tikz-cd}

\let\oldlabel=\label
\def\prellabel{%
  \marginparsep=1em
  \marginparwidth=44pt
  \def\label##1{%
    \oldlabel{##1}%
    \ifmmode\else
      \ifinner\else
        \marginpar{{\footnotesize\ \\ \ttfamily ##1}}%
      \fi
    \fi
  }%
}

\def\opn#1#2{\def#1{\operatorname{#2}}}

\opn\chara{char}
\opn\rank{rank}
\opn\hilb{Hilb}

\opn\gr{gr}
\opn\Rees{\mathcal R}

\newtheorem{theorem}{Theorem}[section]
\newtheorem{definition}[theorem]{Definition}
\newtheorem{cons}[theorem]{Construction}
\newtheorem{lemma}[theorem]{Lemma}
\newtheorem{proposition}[theorem]{Proposition}
\newtheorem{example}[theorem]{Example}
\newtheorem{obs}[theorem]{Observation}
\newtheorem{question}[theorem]{Question}
\newtheorem{remark}[theorem]{Remark}

\newtheorem{conjecture}[theorem]{Conjecture}
\newtheorem{corollary}[theorem]{Corollary}
\newtheorem{setup}[theorem]{Setup}

\newtheorem*{notation*}{Notation}

\let\epsilon=\varepsilon
\let\phi=\varphi
\let\kappa=\varkappa

\newcommand{\vv}{\mathrm{v}}
\newcommand{\pol}{\operatorname{Pol}}
\newcommand{\Min}{\operatorname{Min}}
\newcommand{\depth}{\operatorname{depth}}

\newcommand{\reg}{\operatorname{reg}}

\newcommand{\link}{\operatorname{link}}

\newcommand{\pd}{\operatorname{pd}}

\def\addots{%
  \mathinner{%
    \mkern1mu\raise1pt\hbox{.}%
    \mkern2mu\raise4pt\hbox{.}%
    \mkern2mu\raise7pt\vbox{\kern7pt\hbox{.}}%
    \mkern1mu
  }%
}

\numberwithin{equation}{section}

\title[Powers of Alexander Duals of Chessboard Complexes]
{Powers of Alexander Duals of Stanley--Reisner Ideals of Chessboard Complexes}

\author{Arka Ghosh}
\email{arkaghosh1208@gmail.com}
\address{Department of Mathematics,
Indian Institute of Technology Bhubaneswar,
Bhubaneswar--752050, India}

\author{S. Selvaraja}
\email{selvas@iitbbs.ac.in}
\address{Department of Mathematics,
Indian Institute of Technology Bhubaneswar,
Bhubaneswar--752050, India}



\begin{document}

\begin{abstract}
Let $\Delta_{m,n}$ be the chessboard complex, and let $J_{m,n}$ be the Stanley--Reisner ideal of its Alexander dual. We study the symbolic and ordinary powers of $J_{m,n}$, focusing on their homological and algebraic invariants. First, we characterize when $J_{m,n}^{(q)}$ has linear quotients. For the ordinary powers, we give sufficient conditions on $m$ and $n$ for $J_{m,n}^q$ to have linear quotients. We determine the maximal degree of a minimal generator and the Castelnuovo--Mumford regularity of $J_{m,n}^{(q)}$, obtaining explicit formulas for all $q\ge2$ and showing that both invariants have the same asymptotic slope. We also compute the initial degrees and Waldschmidt constant, proving that $\alpha(J_{m,n}^{(q)})$ is a degree-one quasi-polynomial of period two and $\widehat{\alpha}(J_{m,n})=mn/2$. Explicit formulas for the $\vv$-numbers of symbolic and ordinary powers are obtained. Finally, for every $m\ge3$, we show that $J_{m,2m-1}$ is not weakly polymatroidal, disproving a conjecture of Lu and Wang in this setting.
\end{abstract}
\maketitle

\section{Introduction}

Chessboard complexes and their relatives have been studied extensively in topological and algebraic combinatorics. They arise naturally in group theory, representation theory, commutative algebra, and algebraic topology; see, for example, \cite{J08,M03,Wa03}. Their combinatorial properties have also led to connections with algebraic invariants of associated monomial ideals. In this paper, we investigate the homological and algebraic properties of the Stanley-Reisner ideals of Alexander duals of chessboard complexes, focusing on their symbolic and ordinary powers.

Throughout, we assume $m \le n$. The \emph{chessboard complex} $\Delta_{m,n}$ is the simplicial complex whose vertices are the squares $x_{i,j}$, $1 \le i \le m$, $1 \le j \le n$, and whose faces correspond to placements of non-attacking rooks. Equivalently, $\Delta_{m,n}$ is the matching complex of the complete bipartite graph $K_{m,n}$ and the independence complex of the rook graph $R_{m,n}$. It is pure of dimension $m-1$, since every partial non-attacking placement extends to one with exactly $m$ rooks. The complex $\Delta_{m,n}$ is flag. The relative sizes of $m$ and $n$ play a central role in its structure and will be key to our algebraic analysis.
Garst \cite{garst1979cohen} initiated the study of chessboard complexes, proving that $\Delta_{m,n}$ is Cohen-Macaulay if and only if $n \ge 2m-1$. Ziegler \cite{Z94} strengthened this by showing shellability under the same condition. Since $\Delta_{m,n}$ is pure, this threshold fundamentally governs its homological properties.

Let $I_{\Delta^\vee}$ denote the Stanley--Reisner ideal of the Alexander dual of a simplicial complex $\Delta$. For the chessboard complex, set
\[
J_{m,n} := I_{\Delta_{m,n}^{\vee}} \subseteq R = \mathbb{K}[x_{i,j} \mid 1 \le i \le m,\ 1 \le j \le n],
\]
where $\mathbb{K}$ is a field.
By the Eagon-Reiner theorem \cite{eagon}, $\Delta_{m,n}$ is Cohen-Macaulay if and only if $J_{m,n}$ has a linear resolution. Moreover, $\Delta_{m,n}$ is shellable if and only if $J_{m,n}$ has linear quotients \cite{HerzogsBook}. Thus the threshold $n=2m-1$ is reflected directly in the homological properties of $J_{m,n}$.

This motivates the study of the symbolic and ordinary powers of $J_{m,n}$. For an ideal $I \subseteq R$, let $\Min(I)$ denote the set of minimal primes of $I$. For $q \ge 1$, the $q$-th \emph{symbolic power} of $I$ is
$I^{(q)} = \bigcap_{\mathfrak p \in \Min(I)} \left( I^q R_{\mathfrak p} \cap R \right).$
If $I$ is a squarefree monomial ideal with irredundant prime decomposition $I = \mathfrak p_1 \cap \cdots \cap \mathfrak p_r$, then $I^{(q)} = \mathfrak p_1^q \cap \cdots \cap \mathfrak p_r^q$. Symbolic powers of ideals, particularly their algebraic and homological properties and their comparison with ordinary powers, have been studied extensively; see the survey \cite{DDAGHN}.

We first classify when the symbolic powers of $J_{m,n}$ have linear quotients. If $n \ge 2m-1$, then all symbolic powers have linear quotients. In the range $m < n < 2m-1$, a parity phenomenon occurs: the even symbolic powers have linear quotients, while the odd symbolic powers do not. In the diagonal case $m=n$, none of the symbolic powers has linear quotients. This yields the following classification.

\begin{theorem}[Theorem~\ref{main-classification}]
Let $\Delta = \Delta_{m,n}$ be the chessboard complex. Then, for every $q \ge 1$, the ideal $J_{m,n}^{(q)}$ has linear quotients if and only if one of the following holds:
\begin{enumerate}
    \item $n \ge 2m-1$;
    \item $m < n < 2m-1$ and $q$ is even.
\end{enumerate}
In particular, if $m=n$, then $J_{m,n}^{(q)}$ does not have linear quotients for any $q \ge 1$.
\end{theorem}

For ordinary powers, a complete characterization of the linear-quotient property remains open. We establish the following partial results. When $m=2$ and $n \ge 3$, every ordinary power $J_{2,n}^q$ has linear quotients and hence a linear resolution; see Theorem~\ref{2-ord}. In contrast, for $m=n \ge 2$, the second power $J_{m,m}^2$ does not have a linear resolution; see Theorem~\ref{equal-not}. Moreover, $J_{3,4}^2$ has linear quotients, whereas $J_{4,5}^2$ does not have a linear resolution. These results indicate that the linear-quotient property of ordinary powers does not follow the same simple pattern as that of symbolic powers, and a complete classification remains an open problem; see Question~\ref{ord-que}.

The Castelnuovo--Mumford regularity (or simply regularity), denoted by $\reg(-)$, of symbolic powers of homogeneous ideals has been an active area of research; see \cite{DHNT20} and references therein. For $J_{m,n}$, the regularity is known when $n \ge 2m-1$: in this range,
$\reg(J_{m,n}) = m(n-1),$
since $J_{m,n}$ has a linear resolution. For $m \le n < 2m-1$, Bj\"orner, Lov\'asz, Vre\'cica, and \v{Z}ivaljevi\'c \cite[Corollary~1.4]{BLVZ94} proved
$\reg(J_{m,n}) \le mn - \left\lfloor \frac{m+n+1}{3} \right\rfloor.$
As $\Delta_{m,n}$ is not Cohen--Macaulay in this range, $J_{m,n}$ does not have a linear resolution, so
$\reg(J_{m,n}) \ge m(n-1)+1.$
Hence
$m(n-1)+1 \le \reg(J_{m,n}) \le mn - \left\lfloor \frac{m+n+1}{3} \right\rfloor.$

For a homogeneous ideal $I$, let $\omega(I)$ denote the maximum degree of a minimal generator. We now determine $\omega(J_{m,n}^{(q)})$. While the function $q \mapsto \omega(I^{(q)})$ need not be linear for general monomial ideals (see \cite[Theorem~5.15]{DHNT20}), the symbolic powers of $J_{m,n}$ exhibit a particularly simple behavior. Moreover, for $q \ge 2$, the regularity of $J_{m,n}^{(q)}$ exhibits a dichotomy depending on the board dimensions: in the non-diagonal case $m < n$, the regularity coincides with the maximal generator degree, whereas in the diagonal case $m = n$, it exceeds that degree by one.

\begin{theorem}[Theorem~\ref{deg-sym} and Theorem~\ref{reg-sym}]
For every $q \ge 1$,
\begin{enumerate}
    \item $\omega\!\left(J_{m,n}^{(q)}\right) = qm(n-1).$
    \item $
\reg\!\left(J_{m,n}^{(q)}\right) =
\begin{cases}
qm(n-1), & \text{if } m < n,\\[2mm]
qm(m-1)+1, & \text{if } m = n.
\end{cases}
$
\end{enumerate}

\end{theorem}

These formulas determine the asymptotic invariant $\delta(J_{m,n})$. By \cite{DHNT20},
$\lim_{q \to \infty} \frac{\reg(I^{(q)})}{q} = \lim_{q \to \infty} \frac{\omega(I^{(q)})}{q} = \delta(I).$
For squarefree monomial ideals, $\delta(I) \ge \omega(I)$ \cite[Theorem~3.3, Theorem~3.6 and Lemma~4.3]{DHNT20}. Hence, by Theorems~\ref{deg-sym} and~\ref{reg-sym},
$\delta(J_{m,n}) = m(n-1) = \omega(J_{m,n}).$
Thus $J_{m,n}$ attains the bound $\delta(I) \ge \omega(I)$.
As a consequence, we obtain the following comparison between symbolic and ordinary regularities; see Corollary~\ref{reg-comparision}. In particular,
$\reg\!\left(J_{2,n}^{(q)}\right) = \reg\!\left(J_{2,n}^q\right)$ for all $ q \ge 1,\ n \ge 2,$
and, if $m < n$,
$\reg\!\left(J_{m,n}^{(q)}\right) \le \reg\!\left(J_{m,n}^q\right)$ for all  $q \ge 2.$
For the diagonal case,
$\reg\!\left(J_{m,m}^{(2)}\right) \le \reg\!\left(J_{m,m}^2\right).$

We next study the initial degrees of the symbolic powers. For a homogeneous ideal $I$, let
$\alpha(I) := \min\{\deg(f) \mid 0 \neq f \in I\}$
denote the initial degree of $I$. The asymptotic behavior of these initial degrees is captured by the Waldschmidt constant
\[
\widehat{\alpha}(I) = \lim_{q \to \infty} \frac{\alpha(I^{(q)})}{q}.
\]
For cover ideals of graphs, the Waldschmidt constant can be determined from the second symbolic power; see \cite[Corollary~4.4]{DG20}. For the ideals $J_{m,n}$, we determine $\alpha(J_{m,n}^{(q)})$ for every $q \ge 1$, and hence obtain an explicit formula for the Waldschmidt constant. In particular, the function $q \mapsto \alpha(J_{m,n}^{(q)})$ is a quasi-polynomial of degree one and period two.

\begin{theorem}[Theorem~\ref{min-deg}]
For every $q \ge 1$,
\[
\alpha\!\left(J_{m,n}^{(q)}\right) =
\begin{cases}
\dfrac{mnq}{2}, & \text{if } q \text{ is even},\\[2mm]
\dfrac{mn(q-1)}{2} + m(n-1), & \text{if } q \text{ is odd}.
\end{cases}
\]
Consequently,
$\widehat{\alpha}\!\left(J_{m,n}\right) = \frac{mn}{2}.$
\end{theorem}

We next study the $\vv$-number, an invariant defined in terms of the associated primes of a homogeneous ideal. For a homogeneous ideal $I \subseteq R$, set
\[
\vv(I) = \min\{d \mid \text{ there exists } f \in R_d \text{ such that } I : f \in \operatorname{Ass}_R(R/I)\}.
\]
The $\vv$-number was introduced by Cooper et al.~\cite{CSTPV20} and has since been studied in various algebraic and combinatorial settings; see, for example, \cite{Con24,FS26,KNS25}. Its asymptotic behavior for symbolic powers has also been investigated. By \cite[Theorem~1.3]{FS24}, the function $q \mapsto \vv(I^{(q)})$ is eventually quasi-linear when $I$ is a monomial ideal, while \cite[Corollary~3.6]{KNS25} shows that
$\lim_{q \to \infty} \frac{\vv(I^{(q)})}{q} = \widehat{\alpha}(I).$
Consequently, there exist nonnegative integers $q_1, d_1$ and integers $b_1, \ldots, b_{d_1}$ such that
$\vv(I^{(q)}) = \widehat{\alpha}(I)q + b_i$
for all $q \ge q_1$ satisfying $i \equiv q \pmod{d_1}$. In general, determining the constants $b_i$ and the threshold $q_1$ explicitly can be difficult.
However, for some classes of ideals, these constants have been computed; see, for example, \cite{VP26} for cover ideals of complete graphs and complete bipartite graphs, and \cite{CHJV26} for edge ideals of certain classes of cycles and complete graphs.
For the ideals $J_{m,n}$, we determine these quantities explicitly. In fact,
\[
\vv\!\left(J_{m,n}^{(q)}\right) = \widehat{\alpha}(J_{m,n})q + b_i, \qquad i \equiv q \pmod{2},
\]
for all $q \ge 1$, with $\widehat{\alpha}(J_{m,n}) = \frac{mn}{2}$. More precisely, the constants $b_1$ and $b_2$ are given explicitly as follows (Theorem~\ref{asyv}).
For every $q \ge 1$,
\[
\vv\!\left(J_{m,n}^{(q)}\right) = \widehat{\alpha}(J_{m,n})q + b_i,
\qquad
i =
\begin{cases}
1, & q \text{ odd},\\
2, & q \text{ even},
\end{cases}
\]
where
$\widehat{\alpha}(J_{m,n}) = \frac{mn}{2},$
and
\begin{enumerate}
\item if $m < n$, then
$b_1 = m(n-1)-1-\frac{mn}{2},$
$b_2 = m+n-4;$

\item if $m = n$, then
$b_1 = \frac{m^2-2m}{2},$
$b_2 = 2m-4.$

\end{enumerate}
Equivalently,
\[
\vv\!\left(J_{m,n}^{(q)}\right)
=
\begin{cases}
m(n-1)-1+\dfrac{q-1}{2}mn,
& \text{if } m < n \text{ and } q \text{ is odd},\\[2mm]
m+n-4+\dfrac{q}{2}mn,
& \text{if } m < n \text{ and } q \text{ is even},\\[2mm]
\dfrac{m^2q+m^2-2m}{2},
& \text{if } m = n \text{ and } q \text{ is odd},\\[2mm]
\dfrac{q}{2}m^2+2m-4,
& \text{if } m = n \text{ and } q \text{ is even}.
\end{cases}
\]

The explicit formulas obtained here reveal a clear dichotomy in the asymptotic behavior of the invariants $\alpha$, $\vv$, $\omega$, and $\reg$; a complete comparison is given in Table~\ref{comparison}. We further establish nonvanishing results for certain graded Betti numbers of $R/J_{m,n}^{(q)}$, which are obtained recursively from those of $R/J_{m-1,n-1}^{(q)}$; see Corollary~\ref{betti-nonvanishing}.

For ordinary powers, the situation is more regular. Let $I$ be a homogeneous ideal. Conca \cite{Con24} and, independently, Ficarra and Sgroi \cite{FS26} proved that
\[
\vv(I^q) = \alpha(I)q + b(I)
\]
for some constant $b(I)$ and for all sufficiently large $q$. The constant $b(I)$ has been determined explicitly for several classes of ideals; see, for example, \cite{BMS24}. For the ideals $J_{m,n}$, we obtain an explicit formula for the $\vv$-number of the ordinary powers, valid for all $q \ge 1$.

\begin{theorem}[Theorem~\ref{vnumber-ordinary}]
For every $q \ge 1$,
\[
\vv\!\left(J_{m,n}^q\right) =
\begin{cases}
qm(n-1)-1, & \text{if } m \ne n,\\[2mm]
qm(n-1), & \text{if } m = n.
\end{cases}
\]
\end{theorem}

We compare the $\vv$-numbers of the ordinary powers with their Castelnuovo-Mumford regularity. By Corollary~\ref{vnumber-reg},
$\vv\!\left(J_{m,n}^q\right) \le \reg\!\left(R/J_{m,n}^q\right)$
if $m \ne n$ and $q \ge 1$, and
$\vv\!\left(J_{m,m}^2\right) \le \reg\!\left(R/J_{m,m}^2\right).$
Moreover, the formulas for the $\vv$-numbers of the ordinary and symbolic powers show that their difference is unbounded as $m,n$ vary. More precisely, for every fixed $q \ge 2$, the difference
$\vv\!\left(J_{m,n}^q\right) - \vv\!\left(J_{m,n}^{(q)}\right)$
is unbounded as $m,n$ vary.

We conclude with a result concerning weak polymatroidality. Lu and Wang \cite{LW24} proposed the following conjecture.

\begin{conjecture}[{\cite[Conjecture~1]{LW24}}]
\label{lw-conj}
If $\Delta$ is a pure vertex decomposable simplicial complex, then $I_{\Delta^\vee}$ is weakly polymatroidal.
\end{conjecture}

Lu and Wang verified the conjecture for several classes of simplicial complexes, including independence complexes of cactus, bipartite, and chordal graphs. The chessboard complexes form a natural family to test the conjecture. By \cite{Z94}, $\Delta_{m,n}$ is pure and vertex decomposable precisely when $n \ge 2m-1$. Nevertheless, even in this range, the conjecture fails in general. More precisely, we prove the following.

\begin{theorem}[Theorem~\ref{conter-ex}]
For every $m \ge 3$, the ideal $J_{m,2m-1}$ is not weakly polymatroidal.
\end{theorem}

The paper is organized as follows. Section~\ref{preliminaries} collects the necessary background on graphs, simplicial complexes, monomial ideals, and the homological invariants used throughout. In Section~\ref{sec-lq}, we study the linear-quotient properties of the symbolic and ordinary powers of $J_{m,n}$. Section~\ref{reg-sec} is devoted to the maximal degrees of minimal generators and the Castelnuovo-Mumford regularity of the symbolic powers. In Section~\ref{vnumber-sec}, we determine the initial degrees, the Waldschmidt constant, and the $\vv$-numbers of the symbolic and ordinary powers, as well as certain nonvanishing results for their graded Betti numbers. The final section concerns weak polymatroidality and establishes a counterexample to the conjecture of Lu and Wang.

\section{Notation and Preliminaries}\label{preliminaries}

\subsection{Graphs and simplicial complexes}

Throughout, all graphs are finite and simple. For a graph $G$, we denote its vertex and edge sets by $V(G)$ and $E(G)$, respectively. For $x \in V(G)$, let
$N_G(x) = \{y \in V(G) \mid \{x,y\} \in E(G)\},$
and let $N_G[x] = N_G(x) \cup \{x\}$. For $A \subseteq V(G)$, we write
$N_G[A] = \bigcup_{x \in A} N_G[x].$
The degree of $x$ is denoted by $\deg_G(x)$; thus, $\deg_G(x) = |N_G(x)|$. When the graph is clear from the context, we write $\deg(x)$ for $\deg_G(x)$. If $U \subseteq V(G)$, then $G \setminus U$ denotes the induced subgraph of $G$ on $V(G) \setminus U$. Two graphs $G_1$ and $G_2$ are said to be \emph{isomorphic}, denoted by $G_1 \cong G_2$, if and only if there exists a bijection $\varphi : V(G_1) \to V(G_2)$ such that $\{x,y\} \in E(G_1)$ if and only if $\{\varphi(x),\varphi(y)\} \in E(G_2)$ for all $x,y \in V(G_1)$.
A subset $C \subseteq V(G)$ is called a \emph{vertex cover} of $G$ if every edge of $G$ has at least one endpoint in $C$. A vertex cover $C$ is called \emph{minimal} if no proper subset of $C$ is a vertex cover.

A \emph{simplicial complex} $\Delta$ on a finite vertex set $V(\Delta)$ is a collection of subsets of $V(\Delta)$ such that $F \in \Delta$ and $F' \subseteq F$ imply $F' \in \Delta$. The elements of $\Delta$ are called \emph{faces}, and the maximal faces under inclusion are called \emph{facets}. The dimension of a face $F$ is $\dim F = |F|-1$, and the dimension of $\Delta$ is
$\dim \Delta = \max\{\dim F \mid F \in \Delta\}.$
The simplicial complex $\Delta$ is called \emph{pure} if all its facets have the same dimension.
For a vertex $x \in V(\Delta)$, the \emph{deletion} of $x$ from $\Delta$ is the simplicial complex
$\Delta \setminus x = \{F \in \Delta \mid x \notin F\},$
while the \emph{link} of $x$ in $\Delta$ is
$\link_\Delta(x) = \{F \in \Delta \mid x \notin F,\ F \cup \{x\} \in \Delta\}.$
More generally, for a face $F \in \Delta$, its link is
$\link_\Delta(F) = \{F' \in \Delta \mid F' \cap F = \emptyset,\ F' \cup F \in \Delta\}.$

The \emph{chessboard complex} $\Delta_{m,n}$ is the simplicial complex associated with an $m \times n$ chessboard. Throughout the paper, we assume that $n \ge m$. Let $x_{i,j}$ denote the square in the $i$th row and $j$th column, where $1 \le i \le m$ and $1 \le j \le n$. Thus,
$V(\Delta_{m,n}) = \{x_{i,j} \mid 1 \le i \le m,\ 1 \le j \le n\}.$
A subset $A \subseteq V(\Delta_{m,n})$ is a face of $\Delta_{m,n}$ if and only if no two vertices of $A$ lie in the same row or the same column. Equivalently, the faces of $\Delta_{m,n}$ correspond to placements of non-attacking rooks on the $m \times n$ chessboard. Figure~\ref{fig:chessboard} illustrates this interpretation. In particular, if $x_{1,2} \in A$, then no other vertex in the first row or the second column can belong to $A$. Since $n \ge m$, every facet contains exactly $m$ vertices. Consequently, $\dim \Delta_{m,n} = m-1$.

\begin{figure}[ht]
\centering
\begin{tikzpicture}[scale=0.5]

\draw[thick] (0,0) rectangle (8,8);

\foreach \x in {2,4,6}
    \draw (\x,0)--(\x,8);
\foreach \y in {2,4,6}
    \draw (0,\y)--(8,\y);

\fill[gray!20] (0,6) rectangle (8,8);
\fill[gray!20] (2,0) rectangle (4,8);

\draw[thick] (0,0) rectangle (8,8);
\foreach \x in {2,4,6}
    \draw (\x,0)--(\x,8);
\foreach \y in {2,4,6}
    \draw (0,\y)--(8,\y);

\node at (1,7) {$x_{1,1}$};
\node at (3,7) {$x_{1,2}$};
\node at (5,7) {$\cdots$};
\node at (7,7) {$x_{1,n}$};

\node at (1,5) {$x_{2,1}$};
\node at (3,5) {$x_{2,2}$};
\node at (5,5) {$\cdots$};
\node at (7,5) {$x_{2,n}$};

\node at (1,3) {$\vdots$};
\node at (3,3) {$\ddots$};
\node at (5,3) {$\ddots$};
\node at (7,3) {$\vdots$};

\node at (1,1) {$x_{m,1}$};
\node at (3,1) {$x_{m,2}$};
\node at (5,1) {$\cdots$};
\node at (7,1) {$x_{m,n}$};

\filldraw (3,6.5) circle (3pt);

\end{tikzpicture}
\caption{The vertex $x_{1,2}$ and the squares in the same row or column. The shaded squares cannot belong to a face containing $x_{1,2}$.}
\label{fig:chessboard}
\end{figure}

A simplicial complex $\Delta$ is called \emph{vertex decomposable} if either $\Delta$ is a simplex, or there exists a vertex $x \in V(\Delta)$ such that $\Delta \setminus x$ and $\link_\Delta(x)$ are vertex decomposable and no face of $\link_\Delta(x)$ is a facet of $\Delta \setminus x$. Such a vertex is called a \emph{shedding vertex}.

A simplicial complex $\Delta$ is called \emph{shellable} if its facets can be ordered as $F_1, \ldots, F_t$ such that, for each $k > 1$,
$\langle F_k \rangle \cap \langle F_1, \ldots, F_{k-1} \rangle$
is pure of dimension $\dim(F_k)-1$, where $\langle F_1, \ldots, F_k \rangle$ denotes the simplicial complex generated by $F_1, \ldots, F_k$.

Let $\mathbb{K}$ be a field. A $d$-dimensional simplicial complex $\Delta$ is called \emph{Cohen-Macaulay over $\mathbb{K}$} if
$\widetilde{H}_i(\link_\Delta(F); \mathbb{K}) = 0$
for every $F \in \Delta$ and every $i < d-|F|$. Equivalently, the Stanley-Reisner ring $\mathbb{K}[\Delta]$ is Cohen-Macaulay.

For $i \ge -1$, let $\Delta^{[i]}$ denote the \emph{pure $i$-skeleton} of $\Delta$, that is, the subcomplex generated by the faces of dimension $i$. The simplicial complex $\Delta$ is called \emph{sequentially Cohen-Macaulay over $\mathbb{K}$} if $\Delta^{[i]}$ is Cohen-Macaulay over $\mathbb{K}$ for every $i$ such that $\Delta^{[i]} \ne \emptyset$. We will use the standard implications
\[
\text{vertex decomposable} \Longrightarrow \text{shellable} \Longrightarrow \text{sequentially Cohen-Macaulay}.
\]

For a graph $G$, the \emph{independence complex} of $G$, denoted by $\operatorname{Ind}(G)$, is the simplicial complex whose faces are the independent sets of $G$. We say that $G$ is \emph{vertex decomposable}, \emph{shellable}, or \emph{sequentially Cohen-Macaulay} if $\operatorname{Ind}(G)$ has the corresponding property.

We shall use the following characterization of vertex decomposable graphs.

\begin{lemma}[{\cite[Lemma~4]{Wood2}}]
Let $G$ be a graph. Then $G$ is vertex decomposable if and only if either $G$ has no edges, or there exists $x \in V(G)$ such that
\begin{enumerate}
\item $G \setminus \{x\}$ and $G \setminus N_G[x]$ are vertex decomposable; and
\item no independent set of $G \setminus N_G[x]$ is a maximal independent set of $G \setminus \{x\}$.
\end{enumerate}
\end{lemma}

We shall repeatedly use the following standard inheritance properties of simplicial complexes and their independence complexes.

\begin{theorem}[{\cite{BFH15,VanVilla,Ho77,ProvLouis}}]\label{ind-lm}
Let $\Delta$ be a simplicial complex. If $\Delta$ is vertex decomposable, shellable, or sequentially Cohen-Macaulay, then $\link_{\Delta}(F)$ has the same property for every face $F \in \Delta$.

In particular, let $G$ be a graph and let $A \subseteq V(G)$ be an independent set. If $G$ is sequentially Cohen-Macaulay (respectively, shellable, vertex decomposable), then the induced subgraph $G \setminus N_G[A]$ is sequentially Cohen-Macaulay (respectively, shellable, vertex decomposable). In particular, for every $x \in V(G)$, $G \setminus N_G[x]$ has the same property as $G$.
\end{theorem}

\subsection{Monomial ideals, symbolic powers, and homological invariants}

Let $R = \mathbb{K}[x_1, \ldots, x_N]$ be a standard graded polynomial ring over a field $\mathbb{K}$, and let $M$ be a finitely generated graded $R$-module. We write its minimal graded free resolution as
\[
\mathbf{F} : \quad
0 \longrightarrow F_p \longrightarrow \cdots \longrightarrow F_1 \longrightarrow F_0 \longrightarrow M \longrightarrow 0,
\]
where $F_i = \bigoplus_{j \in \mathbb{Z}} R(-j)^{\beta_{i,j}(M)}$. The integers $\beta_{i,j}(M)$ are the \emph{graded Betti numbers} of $M$; equivalently,
$\beta_{i,j}(M) = \dim_{\mathbb{K}} \operatorname{Tor}_i^R(M, \mathbb{K})_j.$
The \emph{projective dimension} of $M$, denoted by $\pd_R(M)$, is
$\pd_R(M) = \max\{i \mid \beta_{i,j}(M) \ne 0 \text{ for some } j\}.$
The \emph{depth} of $M$, denoted by $\depth_R(M)$, is the length of a maximal $M$-regular sequence contained in the homogeneous maximal ideal $\mathfrak{m} = (x_1, \ldots, x_N)$. Since $R$ is a regular ring of dimension $N$, the Auslander-Buchsbaum formula gives
$\pd_R(M) + \depth_R(M) = N.$
The \emph{Castelnuovo-Mumford regularity}, or simply the \emph{regularity}, of $M$, denoted by $\reg(M)$, is defined by
$\reg(M) = \max\{j-i \mid \beta_{i,j}(M) \ne 0\}.$ For a homogeneous ideal $I \subseteq R$, we use the convention $\reg(R/I) = \reg(I)-1$.

Let $I \subseteq R = \mathbb{K}[x_1, \ldots, x_N]$ be a homogeneous ideal, and denote by $\mathcal{G}(I)$ its minimal homogeneous generating set. Suppose that $I$ is generated in a single degree $d$. We say that $I$ has a \emph{linear resolution} if $\beta_{i,j}(I) = 0$ whenever $j \ne i+d$. In this case, $\reg(I) = d$. We say that $I$ has \emph{linear quotients} if the elements of $\mathcal{G}(I)$ can be ordered as $f_1, \ldots, f_t$ such that $(f_1, \ldots, f_{k-1}) : f_k$ is generated by linear forms for every $k = 2, \ldots, t$.

For each $d$, let $I_{\langle d \rangle}$ denote the ideal generated by the homogeneous elements of $I$ of degree $d$. The ideal $I$ is called \emph{componentwise linear} if $I_{\langle d \rangle}$ has a linear resolution for every $d$.

We shall use the following standard implications; see, for example, \cite{HerzogsBook}:
\[
\text{linear quotients} \Longrightarrow \text{componentwise linear},
\]
and
\[
\text{equigenerated ideal with linear quotients} \Longrightarrow \text{linear resolution} \Longrightarrow \text{componentwise linear}.
\]

Let $\Delta$ be a simplicial complex on the vertex set $V(\Delta) = \{x_1, \ldots, x_N\}$, and let $R = \mathbb{K}[x_1, \ldots, x_N]$. The \emph{Stanley-Reisner ideal} of $\Delta$ is the squarefree monomial ideal
$I_\Delta = \left( x_{i_1} \cdots x_{i_r} \mid \{x_{i_1}, \ldots, x_{i_r}\} \notin \Delta \right) \subseteq R.$
The quotient ring $\mathbb{K}[\Delta] = R/I_\Delta$ is called the \emph{Stanley-Reisner ring} of $\Delta$. The \emph{Alexander dual} of $\Delta$ is the simplicial complex
$\Delta^\vee = \{V(\Delta) \setminus F \mid F \notin \Delta\}.$
We write $I_{\Delta^\vee}$ for the Stanley-Reisner ideal of the Alexander dual.

\begin{definition}
The \emph{rook graph} $R_{m,n}$ associated with an $m \times n$ chessboard is the graph with vertex set
$V(R_{m,n}) = \{x_{i,j} \mid 1 \le i \le m,\ 1 \le j \le n\},$
where two distinct vertices $x_{i,j}$ and $x_{k,\ell}$ are adjacent if and only if $i=k$ or $j=\ell$. Thus, two vertices are adjacent precisely when the corresponding squares of the chessboard lie in a common row or a common column.
\end{definition}

\begin{obs}
The chessboard complex $\Delta_{m,n}$ is a flag complex. Indeed, $\Delta_{m,n} = \operatorname{Ind}(R_{m,n})$, and hence its Stanley-Reisner ideal is generated by quadratic monomials:
$I_{\Delta_{m,n}} = I(R_{m,n}).$
\end{obs}

We now recall the notion of the cover ideal of a graph.

\begin{definition}\label{def-cover-ideal}
For a graph \(G\), we write
$J(G) := I_{\operatorname{Ind}(G)^\vee}.$
This is the \emph{cover ideal} of \(G\). Equivalently,
\[
J(G) = \bigcap_{\{x_i,x_j\} \in E(G)} (x_i, x_j),
\]
and its minimal monomial generators correspond precisely to the minimal vertex covers of \(G\).
\end{definition}

We shall use the following fundamental results relating combinatorial properties of a simplicial complex to homological properties of the Stanley-Reisner ideal of its Alexander dual.

\begin{theorem}[Eagon-Reiner, \cite{eagon}; Terai, \cite{Terai00}; \cite{HerzogsBook}]
\label{ER-Terai}
Let $\Delta$ be a simplicial complex on the vertex set $\{x_1, \ldots, x_N\}$, and let $R = \mathbb{K}[x_1, \ldots, x_N]$. Then
\begin{enumerate}
    \item $\mathbb{K}[\Delta]$ is Cohen-Macaulay if and only if $I_{\Delta^\vee}$ has a linear resolution;

    \item $\pd_R(\mathbb{K}[\Delta]) = \reg\!\left(I_{\Delta^\vee}\right)$;

    \item $\Delta$ is shellable if and only if $I_{\Delta^\vee}$ has linear quotients.
\end{enumerate}
\end{theorem}

We recall the following characterization of the chessboard complex. As a consequence of \cite{Z94}, \cite{garst1979cohen}, and Theorem~\ref{ER-Terai}, we have the following.

\begin{theorem}\label{base-known}
Let $\Delta_{m,n}$ be the chessboard complex. Then $J_{m,n}$ has linear quotients if and only if $n \ge 2m-1$.
\end{theorem}

\section{Linear quotients of symbolic and ordinary powers}\label{sec-lq}

In this section, we study the linear-quotient property of the symbolic and ordinary powers of $J_{m,n}$. We first give a complete characterization for symbolic powers and then establish positive and negative results for ordinary powers.

\subsection{Symbolic powers}

We begin by recalling the polarization of a monomial ideal. Let $R = \mathbb{K}[x_1, \ldots, x_n]$ be a polynomial ring over a field $\mathbb{K}$, and let $M = x_1^{a_1} \cdots x_n^{a_n}$ be a monomial in $R$. The \emph{polarization} of $M$ is the squarefree monomial
$M^{\pol} = \prod_{i=1}^n \prod_{j=1}^{a_i} x_{ij}$
in the polynomial ring
$R' = \mathbb{K}[x_{ij} \mid 1 \le i \le n,\ 1 \le j \le a_i].$
If $I = (M_1, \ldots, M_r) \subseteq R$ is a monomial ideal, its polarization is the squarefree monomial ideal
$I^{\pol} = (M_1^{\pol}, \ldots, M_r^{\pol}) \subseteq R'.$

We next recall a construction introduced in \cite{Fakhari}. The author constructed it for an arbitrary graph, but we focus here on the rook graph.

\begin{cons}[\cite{Fakhari}]
Let $R_{m,n}$ be the rook graph. For an integer $q \ge 1$, define the graph $R_{m,n}^q$ by
$V(R_{m,n}^q) = \{x_{i,j,p} \mid 1 \le i \le m,\ 1 \le j \le n,\ 1 \le p \le q\},$
and
\[
E(R_{m,n}^q) = \bigl\{ \{x_{i,j,p}, x_{i',j',p'}\} \mid \{x_{i,j}, x_{i',j'}\} \in E(R_{m,n}),\ p+p' \le q+1 \bigr\}.
\]
\end{cons}
For $q=1$, we have $R_{m,n}^1\cong R_{m,n}$.

The following lemma gives the polarization correspondence that will be used below.

\begin{lemma}\label{lq-cl}
Let $\Delta_{m,n}$ be the chessboard complex. Then, for every $q \ge 1$,
$\left(J_{m,n}^{(q)}\right)^{\pol} = J(R_{m,n}^q).$
Moreover, $J(R_{m,n}^q)$ is componentwise linear (respectively, has linear quotients) if and only if $J_{m,n}^{(q)}$ is componentwise linear (respectively, has linear quotients).
\end{lemma}

\begin{proof}
The equality $\left(J_{m,n}^{(q)}\right)^{\pol} = J(R_{m,n}^q)$ is proved in \cite[Lemma~3.4]{Fakhari}. The remaining assertions follow from \cite[Lemma~3.5]{Fakhari} and \cite[Proposition~1]{NPY21}.
\end{proof}

For $1 \le k \le m-1$, let $G_{k,m,n}^q$ denote the induced subgraph of $R_{m,n}^q$ on the vertex set
\[
V(G_{k,m,n}^q) = V(R_{m,n}^q) \setminus \{x_{i,j,1} \mid 1 \le i \le m-k,\ 1 \le j \le n\},
\]
where $q \ge 2$ and $m \ge 2$.

\begin{lemma}\label{tech-lm}
Suppose that $R_{m,n}^t$ is vertex decomposable for every $1 \le t < q$ and every pair $(m,n)$ satisfying $n \ge 2m-1$. Then $G_{s,m,n}^q$ is vertex decomposable for every $1 \le s \le m-1$.
\end{lemma}
\begin{proof}
We proceed by induction on $s$. First, let $s=1$ and set $G = G_{1,m,n}^q$. For $i = 1, \ldots, n$, let $x_i = x_{m,i,1}$. Since $m \ge 2$ and $n \ge 2m-1$, we have $n \ge 3$. The vertex $x_1$ has the simplicial neighbor $x_{1,1,q}$ and is therefore a shedding vertex \cite[Corollary~7]{Wood2}. Moreover,
$G \setminus N_G[x_1] \cong R_{m-1,n-1}^{q-2},$
up to isolated vertices. Since $n-1 \ge 2(m-1)-1$ and $R_{m,n}^t$ is vertex decomposable for every $1 \le t < q$ and every pair $(m,n)$ satisfying $n \ge 2m-1$, it follows that $G \setminus N_G[x_1]$ is vertex decomposable. Thus it suffices to show that $G \setminus x_1$ is vertex decomposable.

One can observe that $x_2$ has a simplicial neighbor $x_{1,2,q}$ in $G \setminus x_1$, and $x_3$ has a simplicial neighbor $x_{1,3,q}$ in $G \setminus \{x_1, x_2\}$, and so on. Also, $G \setminus N_G[x_i] \cong R_{m-1,n-1}^{q-2}$ up to isolated vertices. So the same argument applies successively to $x_2, \ldots, x_n$. Therefore, it is enough to consider
$G \setminus \{x_1, \ldots, x_n\} \cong R_{m,n}^{q-2},$
which is vertex decomposable by the hypothesis of the theorem. Hence $G_{1,m,n}^q$ is vertex decomposable.

Now assume that $s \ge 2$ and that $G_{i,m,n}^q$ is vertex decomposable for every $1 \le i < s$. Set $G = G_{s,m,n}^q$ and $x_i = x_{m-s+1,i,1}$, $i = 1, \ldots, n$. The vertex $x_1$ has the simplicial neighbor $x_{1,1,q}$ and is therefore a shedding vertex. Furthermore,
$G \setminus N_G[x_1] \cong G_{s-1,m-1,n-1}^q.$
Since $n-1 \ge 2(m-1)-1$, the induction hypothesis on $s$ implies that $G \setminus N_G[x_1]$ is vertex decomposable. As above, it remains to show that $G \setminus x_1$ is vertex decomposable. Applying the same argument successively to $x_2, \ldots, x_n$, it suffices to observe that
$G \setminus \{x_1, \ldots, x_n\} \cong G_{s-1,m,n}^q,$
which is vertex decomposable by the induction hypothesis on $s$. Thus $G_{s,m,n}^q$ is vertex decomposable.
\end{proof}

We begin with a vertex-decomposability result that will be used below.

\begin{lemma}\label{tech-lm0}
Suppose $R_{m,n}^{2t}$ is vertex decomposable for $1 \le t \le q-1$ and for all $(m,n)$ satisfying $m < n$. Then $G_{s,m,n}^{2q}$ is vertex decomposable for every $1 \le s \le m-1$.
\end{lemma}

\begin{proof}
We can prove the result by induction on $s$, proceeding similarly to Lemma~\ref{tech-lm}.
\end{proof}

\begin{theorem}\label{R^{2q}_{m,n}}
Let $m < n$. Then $R_{m,n}^{2q}$ is vertex decomposable for every $q \ge 1$.
\end{theorem}

\begin{proof}
We argue by induction on $m+q$. The case $m=1$ and $q=1$ is immediate. Assume that $R_{a,b}^{2r}$ is vertex decomposable whenever $a+r < m+q$ and $a < b$. Set $G = R_{m,n}^{2q}$ and $x_1 = x_{1,1,1}$. Then
$G \setminus N_G[x_1] \cong R_{m-1,n-1}^{2q},$
which is vertex decomposable by the induction hypothesis, since $(m-1)+q < m+q$ and $m-1 < n-1$.

We claim that $x_1$ is a shedding vertex. Let $I$ be an independent set of $G \setminus N_G[x_1]$. If $x_{1,1,2q} \notin I$, then $I \cup \{x_{1,1,2q}\}$ is independent in $G \setminus \{x_1\}$. Suppose that $x_{1,1,2q} \in I$. If $I \cup \{x_{1,j,2q}\}$ is not independent for every $2 \le j \le n$, then, for each $j \in \{2, \ldots, n\}$, there exists $b_j \in \{2, \ldots, m\}$ such that $x_{b_j,j,1} \in I$. Since $n-1 > m-1$, the pigeonhole principle gives distinct $i,j \in \{2, \ldots, n\}$ with $b_i = b_j$. But then $x_{b_i,i,1}$ and $x_{b_j,j,1}$ are adjacent in $G$, contradicting the independence of $I$. Hence, for some $j \in \{2, \ldots, n\}$, $I \cup \{x_{1,j,2q}\}$ is independent in $G \setminus \{x_1\}$. Thus $x_1$ is a shedding vertex.

Applying the same argument successively to $x_{1,2,1}, \ldots, x_{1,n,1}$, it remains to show that $G \setminus \{x_{1,1,1}, \ldots, x_{1,n,1}\}$ is vertex decomposable. By construction,
$G \setminus \{x_{1,1,1}, \ldots, x_{1,n,1}\} \cong G_{m-1,m,n}^{2q},$
and hence the assertion follows from Lemma~\ref{tech-lm0}. Therefore, $R_{m,n}^{2q}$ is vertex decomposable.
\end{proof}

The following proposition treats the diagonal case $m=n$ and provides an obstruction to linear quotients.

\begin{proposition}\label{R^q_{m,m}}
For every $m \ge 2$ and $q \ge 1$, the ideal $J_{m,m}^{(q)}$ is not componentwise linear. In particular, $J_{m,m}^{(q)}$ does not have linear quotients.
\end{proposition}

\begin{proof}
By Lemma~\ref{lq-cl}, it suffices to show that $J(R_{m,m}^q)$ is not componentwise linear for every $q \ge 1$. We first consider the cases $q=1$ and $q=2$.

For $q=1$, we show by induction on $m$ that $R_{m,m}^1$ is not sequentially Cohen-Macaulay. For $m=2$, the graph $R_{2,2}^1$ is a cycle of length $4$. By \cite[Theorem~10]{Wood2}, $R_{2,2}^1$ is not sequentially Cohen-Macaulay. Hence $J(R_{2,2}^1)$ is not componentwise linear. Now let $m>2$ and suppose, to the contrary, that $R_{m,m}^1$ is sequentially Cohen-Macaulay. By Theorem~\ref{ind-lm}, $R_{m,m}^1 \setminus N_{R_{m,m}^1}[x_{1,1,1}]$ is sequentially Cohen-Macaulay. Since
$R_{m,m}^1 \setminus N_{R_{m,m}^1}[x_{1,1,1}] \cong R_{m-1,m-1}^1,$
this contradicts the induction hypothesis. Thus $R_{m,m}^1$ is not sequentially Cohen-Macaulay, and consequently $J(R_{m,m}^1)$ is not componentwise linear.

The case $q=2$ is similar. We again proceed by induction on $m$. For $m=2$, the graph $R_{2,2}^2$ is bipartite. By \cite[Corollary~2.10 and Theorem~3.10]{VanVilla}, $R_{2,2}^2$ is not sequentially Cohen-Macaulay. Consequently, $J(R_{2,2}^2)$ is not componentwise linear. Suppose that $m>2$ and, to the contrary, that $R_{m,m}^2$ is sequentially Cohen-Macaulay. By Theorem~\ref{ind-lm}, $R_{m,m}^2 \setminus N_{R_{m,m}^2}[x_{1,1,1}]$ is sequentially Cohen-Macaulay. On the other hand, by \cite[Lemma~3.3]{selva1},
$R_{m,m}^2 \setminus N_{R_{m,m}^2}[x_{1,1,1}] \cong R_{m-1,m-1}^2 \cup \{\text{isolated vertices}\}.$
This contradicts the induction hypothesis. Hence $R_{m,m}^2$ is not sequentially Cohen-Macaulay, and therefore $J(R_{m,m}^2)$ is not componentwise linear.

We have thus established that $J(R_{m,m}^q)$ is not componentwise linear for $q=1,2$. By Lemma~\ref{lq-cl}, it follows that $J_{m,m}^{(q)}$ is not componentwise linear for $q=1,2$. Finally, \cite[Theorem~3.2]{SS23} implies that $J_{m,m}^{(q)}$ is not componentwise linear for every $q \ge 1$. Since every ideal with linear quotients is componentwise linear, $J_{m,m}^{(q)}$ does not have linear quotients.
\end{proof}

The following theorem gives a complete characterization of the linear-quotient property for the symbolic powers of the ideals $J_{m,n}$.

\begin{theorem}\label{main-classification}
Let $\Delta_{m,n}$ be the chessboard complex. Then, for every $q \ge 1$, the ideal $J_{m,n}^{(q)}$ has linear quotients if and only if one of the following conditions holds:
\begin{enumerate}
\item $n \ge 2m-1$;
\item $m < n < 2m-1$ and $q$ is even.
\end{enumerate}
In particular, if $m=n$, then $J_{m,m}^{(q)}$ does not have linear quotients for any $q \ge 1$.
\end{theorem}

\begin{proof}
Suppose first that $n \ge 2m-1$. We claim that $R_{m,n}^q$ is vertex decomposable for every $q \ge 1$. We proceed by induction on $m+q$.

The case $m=1$ and $q=1$ is immediate, since $R_{1,n}^1 \cong K_n$ is vertex decomposable. Suppose that $R_{a,b}^t$ is vertex decomposable whenever $a+t < m+q$ and $b \ge 2a-1$. Set $G = R_{m,n}^q$ and $x_1 = x_{1,1,1}$. Then
$G \setminus N_G[x_1] \cong R_{m-1,n-1}^q \cup \{\text{isolated vertices}\}.$
Since $n-1 \ge 2(m-1)-1$ and $(m-1)+q < m+q$, the induction hypothesis shows that $G \setminus N_G[x_1]$ is vertex decomposable.

Let $I$ be an independent set of $G \setminus N_G[x_1]$. By the argument in the proof of Theorem~\ref{R^{2q}_{m,n}}, there exists $i \in \{2, \ldots, n\}$ such that $I \cup \{x_{1,i,q}\}$ is independent in $G \setminus \{x_1\}$. Hence, $x_1$ is a shedding vertex. Applying the same argument successively to $x_{1,2,1}, \ldots, x_{1,n,1}$, it remains to show that
$G \setminus \{x_{1,1,1}, \ldots, x_{1,n,1}\}$
is vertex decomposable. By construction,
$G \setminus \{x_{1,1,1}, \ldots, x_{1,n,1}\} \cong G_{m-1,m,n}^q.$
Thus, Lemma~\ref{tech-lm} implies that this graph is vertex decomposable. Consequently, $R_{m,n}^q$ is vertex decomposable, and Lemma~\ref{lq-cl} yields that $J_{m,n}^{(q)}$ has linear quotients.

Next, suppose that $m < n < 2m-1$. By Theorem~\ref{R^{2q}_{m,n}}, $R_{m,n}^{2q}$ is vertex decomposable for every $q \ge 1$. Hence, Lemma~\ref{lq-cl} implies that $J_{m,n}^{(2q)}$ has linear quotients for every $q \ge 1$.

Conversely, suppose that $J_{m,n}^{(q)}$ has linear quotients for some fixed $q \ge 1$. By Proposition~\ref{R^q_{m,m}}, we have $m \ne n$; under the standing assumption $m \le n$, this gives $m < n$. If $q=1$, then Theorem~\ref{base-known} yields $n \ge 2m-1$. Thus, assume that $q \ge 2$. If $n \ge 2m-1$, there is nothing to prove. If $n < 2m-1$, then $J_{m,n}$ does not have a linear resolution and hence is not componentwise linear. Hence, by \cite[Lemma~3.1]{SS23}, $J_{m,n}^{(q)}$ is not componentwise linear for every odd $q \ge 1$. Since an ideal with linear quotients is componentwise linear, our assumption that $J_{m,n}^{(q)}$ has linear quotients implies that $q$ is even. Thus, $m < n < 2m-1$ and $q$ is even. This completes the proof.
\end{proof}

\subsection{Ordinary powers}

We now turn to the ordinary powers of the ideals $J_{m,n}$. We first consider the case $m=2$. We shall use the notion of \emph{weakly polymatroidal ideals}, introduced by Kokubo and Hibi \cite{KK06}.
Let $I \subseteq R = \mathbb{K}[x_1, \ldots, x_N]$ be a monomial ideal. We say that $I$ is \emph{weakly polymatroidal} if, for any $f,g \in \mathcal{G}(I)$, where
$f = x_1^{a_1} \cdots x_N^{a_N}$ and $g = x_1^{b_1} \cdots x_N^{b_N},$
if, for some $t$, one has $a_i = b_i$ for $1 \le i < t$ and $a_t > b_t$, then there exists $\ell > t$ such that
$x_t \frac{g}{x_\ell} \in I.$ For a monomial  $m$ in the polynomial ring  $R$, let  $\operatorname{supp}(m)$ denote the set of variables dividing $m$.

\begin{theorem}\label{2-ord}
For every $n \ge 3$ and $q \ge 1$, the ideal $J_{2,n}^q$ has linear quotients.
\end{theorem}

\begin{proof}
We show that $J_{2,n}^q$ is weakly polymatroidal. Order the variables as follows:
\[
x_{1,1} < x_{2,2} < x_{2,3} < \cdots < x_{2,n} < x_{2,1} < x_{1,2} < x_{1,3} < \cdots < x_{1,n}.
\]
Let $u = x_{1,1}^{a_{1,1}} \cdots x_{2,n}^{a_{2,n}}$ and $v = x_{1,1}^{b_{1,1}} \cdots x_{2,n}^{b_{2,n}}$ be two minimal generators of $J_{2,n}^q$. Suppose first that
$a_{1,1} = b_{1,1}, \ldots, a_{1,t-1} = b_{1,t-1}, ~ a_{1,t} > b_{1,t}.$

Since $\deg u = \deg v$, there exists a variable $x > x_{1,t}$ dividing $v$. Moreover, since $a_{1,t} > b_{1,t}$, some minimal generator $n_1 \mid v$ of $J_{2,n}$ is not divisible by $x_{1,t}$. Notice that $x_{1,t} \ne x_{1,n}$, since otherwise $\deg u = \deg v$ would imply $a_{1,n} = b_{1,n}$, contradicting $a_{1,t} > b_{1,t}$.
Next, suppose that
$a_{1,1} = b_{1,1}, a_{2,2} = b_{2,2}, \ldots, a_{2,t-1} = b_{2,t-1}, ~ a_{2,t} > b_{2,t}.$
We first show that $x_{2,t} \ne x_{2,1}$. Suppose, to the contrary, that $x_{2,t} = x_{2,1}$. Since $a_{2,1} > b_{2,1}$, there exists a minimal generator $n_1 \mid v$ of $J_{2,n}$ such that $x_{2,1} \nmid n_1$. Write
$u = m_1 \cdots m_q,$ $v = n_1 \cdots n_q,$
where $m_i, n_i \in \mathcal{G}(J_{2,n})$ for all $i$. Suppose that $x_{2,1} \mid m_1, \ldots, m_i$ and $x_{2,1} \mid n_1, \ldots, n_j$. Since $a_{2,1} > b_{2,1}$, we have $i > j$. Let
$E_1 = \{x_{2,2}, \ldots, x_{2,n}\}.$ Then
$\sum_{x \in E_1} \nu_x(v) = j(n-2) + (q-j)(n-1),$ 
whereas
$\sum_{x \in E_1} \nu_x(u) = i(n-2) + (q-i)(n-1),$ where $\nu_x(m)$ is the exponent of $x$ in $m.$
Since $i > j$,
$ \sum_{x \in E_1} \nu_x(v)>\sum_{x \in E_1} \nu_x(u),$
contradicting $a_{2,2} = b_{2,2} = \cdots = a_{2,n}$. Hence $x_{2,t} \ne x_{2,1}$.

Now suppose that
$a_{1,1} = b_{1,1}, \ldots, a_{1,n-2} = b_{1,n-2}, ~ a_{1,n-1} > b_{1,n-1}.$
Since $\deg u = \deg v$, we have $a_{1,n} < b_{1,n}$. Suppose that $x_{1,n} \mid n_1, \ldots, n_j$ and $x_{1,n} \mid m_1, \ldots, m_i$. Then $i < j$ and $x_{2,n}\mid n_{j+1},\ldots,n_q$
and $x_{2,n}\mid m_{i+1},\ldots,m_q.$ Since $a_{2,n} = b_{2,n}$, there exists a minimal generator $n' \mid v$ of $J_{2,n}$ such that $x_{1,n}, x_{2,n} \mid n'$. Choose a factor $n_1 \mid v$ such that $x_{1,n-1} \nmid n_1$.

If $n_1 = n'$, then
$x_{1,n-1} \frac{n'}{x_{1,n}} \in \mathcal{G}(J_{2,n}),$
and hence
$x_{1,n-1} \frac{v}{x_{1,n}} \in J_{2,n}^q.$
Assume that $n_1 \ne n'$. Then
$n_1 = x_{1,1} \cdots x_{1,n-2} x_{1,n} x_{2,1} \cdots x_{2,n-1},$
and
$$n' = x_{1,1} \cdots x_{1,i-1} x_{1,i+1} \cdots x_{1,n} x_{2,1} \cdots x_{2,j-1} x_{2,j+1} \cdots x_{2,n}.$$
Define
$n_1' = x_{1,1} \cdots x_{1,n-1} x_{2,1} \cdots x_{2,j-1} x_{2,j+1} \cdots x_{2,n}$
and
$$n'' = x_{1,1} \cdots x_{1,i-1} x_{1,i+1} \cdots x_{1,n} x_{2,1} \cdots x_{2,j} \cdots x_{2,n-1}.$$
Then
$x_{1,n-1} \frac{n_1}{x_{1,n}} \cdot n' = n_1' n'' \in \mathcal{G}(J_{2,n}^2),$
and therefore
$x_{1,n-1} \frac{v}{x_{1,n}} \in J_{2,n}^q.$

More generally, suppose that
$a_{1,1} = b_{1,1}, \ldots, a_{1,t-1} = b_{1,t-1}, ~ a_{1,t} > b_{1,t}.$
If there exist $t_1, t_2$ such that $x_{1,t_1}, x_{1,t_2} > x_{1,t}$, then there exists a factor $n_1 \mid v$ such that $x_{1,t} \nmid n_1$. Since either $x_{2,t_1} \mid n_1$ or $x_{2,t_2} \mid n_1$, we obtain
$x_{1,t} \frac{n_1}{x_{1,t_i}} \in \mathcal{G}(J_{2,n})$
for some $i \in \{1,2\}$. Consequently,
$x_{1,t} \frac{v}{x_{1,t_i}} \in J_{2,n}^q.$
The argument for the second row is analogous. Thus, suppose that $a_{2,t} > b_{2,t}$. If there exist $t_1, t_2$ such that $x_{2,t_1}, x_{2,t_2} > x_{2,t}$, then the same argument applies. So let $a_{1,1}=b_{1,1},\dots ,a_{2,n-1}=b_{2,n-1} \text{ and } a_{2,n}>b_{2,n}.$ So $x_{2,t} = x_{2,n}$. Choose a factor $n_1 \mid v$ such that $x_{2,n} \nmid n_1$. Then $x_{2,1} \mid n_1$. If $x_{1,1} \mid n_1$, then
$x_{2,n} \frac{n_1}{x_{2,1}} \in \mathcal{G}(J_{2,n}),$
and the required exchange follows. Hence we may assume that $x_{1,1} \nmid n_1$. Suppose there exists a factor $n_2 \mid v$ such that $x_{1,1}, x_{2,1} \mid n_2$. Now let $n_1 = x_{1,2} \cdots x_{1,n} x_{2,1} \cdots x_{2,n-1}$ and $n_2 = x_{1,1} \cdots x_{1,i-1} x_{1,i+1} \cdots x_{1,n} x_{2,1} \cdots x_{2,j-1} x_{2,j+1} \cdots x_{2,n}$. Also let \\
$n_1' = x_{1,1} \cdots x_{1,i-1} x_{1,i+1} \cdots x_{1,n} x_{2,2} \cdots x_{2,n}$ and $$n_2' = x_{1,2} \cdots x_{1,n} x_{2,1} \cdots x_{2,j-1} x_{2,j+1} \cdots x_{2,n}.$$ Then $x_{2,n} \frac{n_1 n_2}{x_{2,1}} = n_1' n_2' \in \mathcal{G}(J_{2,n}^2)$. So $x_{2,n} \frac{v}{x_{2,1}} \in J_{2,n}^q$. We now prove the existence of such a factor. So, for a contradiction, assume that there is no factor $n_i \mid v$ such that $x_{1,1}, x_{2,1} \mid n_i$.

We now show that $a_{2,1} < b_{2,1}$. Suppose, to the contrary, that $a_{2,1} > b_{2,1}$. Let $S_1 = \{x_{2,2}, \ldots, x_{2,n}\}$. Suppose that $x_{2,1} \mid m_1, \ldots, m_i$ and $x_{2,1} \mid n_1, \ldots, n_j$. Then $i > j$. Consequently,
$\sum_{x \in S_1} \nu_x(v) = (q-j)(n-1) + j(n-2).$
Since $a_{2,n} > b_{2,n}$,
$\sum_{x \in S_1} \nu_x(u) \ge (q-j)(n-1) + j(n-2) + 1.$
On the other hand,
\[
\sum_{x \in S_1} \nu_x(u) = (q-i)(n-1) + i(n-2) < (q-j)(n-1) + j(n-2),
\]
a contradiction.

It remains to exclude the case $a_{2,1} = b_{2,1}$. Let
$S' = S_1 \setminus \{x_{2,n}\}.$
Suppose that $x_{2,1} \mid m_1, \ldots, m_i$ and $x_{2,1} \mid n_1, \ldots, n_i$. Then $x_{2,n} \mid m_{i+1}, \ldots, m_q$ and $x_{2,n} \mid n_{i+1}, \ldots, n_q$. Assume further that $x_{2,n} \mid n_j, \ldots, n_i$ and $x_{2,n} \mid m_k, \ldots, m_i$. Since $a_{2,n} > b_{2,n}$, we have $k < j$. Hence
$\sum_{x \in S'} \nu_x(u) < \sum_{x \in S'} \nu_x(v),$
contradicting $a_{2,2} = b_{2,2}, \ldots, a_{2,n-1} = b_{2,n-1}$. Therefore, $a_{2,1} < b_{2,1}$.

Finally, let $x_{2,1} \mid m_1, \ldots, m_i$ and $x_{2,1} \mid n_1, \ldots, n_j$. Then $i < j$. Hence $x_{1,1} \mid m_{i+1}, \ldots, m_q$ and $x_{1,1} \mid n_{j+1}, \ldots, n_q$. Therefore,
$a_{1,1} \ge q-i > q-j = b_{1,1},$
contradicting $a_{1,1} = b_{1,1}$. Hence there exists a factor $n_2 \mid v$ such that $x_{1,1}, x_{2,1} \mid n_2,$ which is a contradiction. So there exists a minimal generator $m\mid v$ of $\mathcal{G}(J_{2,n})$ such that $x_{1,1}, x_{2,1} \mid m.$ Thus, $J_{2,n}^q$ is weakly polymatroidal. By \cite[Theorem~1.3]{FS10}, every weakly polymatroidal ideal has linear quotients. Hence $J_{2,n}^q$ has linear quotients.
\end{proof}

\begin{obs}\label{cover-pol}
Let $J = J(R_{m,n})$ be the cover ideal of the rook graph $R_{m,n}$. By \cite[Theorem~3.2]{FHV10},
\[
J^2 = \bigcap_{\{x_{i,j},x_{k,\ell}\} \in E(R_{m,n})} (x_{i,j}, x_{k,\ell})^2 \cap \bigcap_{\substack{\{x_{r,s},x_{r,u},x_{r,v}\} \\ \text{triangle of } R_{m,n}}} (x_{r,s}^2, x_{r,u}^2, x_{r,v}^2) \cap \bigcap_{\substack{\{x_{e,f},x_{g,f},x_{h,f}\} \\ \text{triangle of } R_{m,n}}} (x_{e,f}^2, x_{g,f}^2, x_{h,f}^2).
\]
Since polarization commutes with finite intersections \cite[Proposition~2.3]{faridi}, we obtain
\[
\begin{aligned}
(J^2)^{\pol} ={}& \bigcap_{\{x_{i,j},x_{k,\ell}\} \in E(R_{m,n})} \bigl((x_{i,j}, x_{k,\ell})^2\bigr)^{\pol} 
\cap \bigcap_{\substack{\{x_{r,s},x_{r,u},x_{r,v}\} \\ \text{triangle of } R_{m,n}}} (x_{r,s}^2, x_{r,u}^2, x_{r,v}^2)^{\pol} \\
&\cap \bigcap_{\substack{\{x_{e,f},x_{g,f},x_{h,f}\} \\ \text{triangle of } R_{m,n}}} (x_{e,f}^2, x_{g,f}^2, x_{h,f}^2)^{\pol}.
\end{aligned}
\]
For an edge $\{x_{i,j}, x_{k,\ell}\}$, we have
$\bigl((x_{i,j}, x_{k,\ell})^2\bigr)^{\pol} = (x_{i,j,1} x_{i,j,2},\ x_{i,j,1} x_{k,\ell,1},\ x_{k,\ell,1} x_{k,\ell,2}).$
For a triangle $\{x_{r,s}, x_{r,u}, x_{r,v}\}$,
\[
(x_{r,s}^2, x_{r,u}^2, x_{r,v}^2)^{\pol} = \bigcap_{a,b,c \in \{1,2\}} (x_{r,s,a}, x_{r,u,b}, x_{r,v,c}),
\]
and similarly for a triangle in a column. Every prime in these latter intersections, except $(x_{r,s,2}, x_{r,u,2}, x_{r,v,2})$ and $(x_{e,f,2}, x_{g,f,2}, x_{h,f,2})$, contains one of the primes arising from the polarization of an edge factor and is therefore redundant. Consequently, $(J^2)^{\pol} = J(\mathcal{H}_{m,n})$, where $\mathcal{H}_{m,n}$ is the simple hypergraph with vertex set
$V(\mathcal{H}_{m,n}) = \{x_{i,j,\ell} \mid 1 \le i \le m,\ 1 \le j \le n,\ \ell \in \{1,2\}\},$
and edge set
\[
\begin{aligned}
E(\mathcal{H}_{m,n}) ={}& \Bigl\{ \{x_{i,j,\ell}, x_{i',j',\ell'}\} \mid \ell + \ell' \le 3,\ \{x_{i,j}, x_{i',j'}\} \in E(R_{m,n}) \Bigr\} \\
&\cup \Bigl\{ \{x_{i,j,2}, x_{i',j',2}, x_{i'',j'',2}\} \mid R_{m,n}[\{x_{i,j}, x_{i',j'}, x_{i'',j''}\}] \cong K_3 \Bigr\},
\end{aligned}
\]
where $K_3$ denotes the complete graph on $3$ vertices.
\end{obs}
Let $\mathcal{H}$ be a hypergraph. We recall that the contraction of a vertex $v$, denoted by $\mathcal{H}/v$, is the hypergraph with vertex set
$V(\mathcal{H}/v) = V(\mathcal{H}) \setminus \{v\}$
and edge set
$E(\mathcal{H}/v) = \min\{e \setminus \{v\} \mid e \in E(\mathcal{H})\},$
where $\min$ denotes the set of minimal edges under inclusion.
The next result gives an obstruction to a linear resolution in the square case.

\begin{theorem}\label{equal-not}
For every $m \ge 2$, the ideal $J_{m,m}^2$ does not have a linear resolution. In particular, $J_{m,m}^2$ does not have linear quotients.
\end{theorem}

\begin{proof}
By Observation~\ref{cover-pol},
$(J_{m,m}^2)^{\pol} = J(\mathcal{H}_{m,m}).$
We first show that the independence complex of $\mathcal{H}_{m,m}$ is not sequentially Cohen--Macaulay. Suppose, to the contrary, that $\mathcal{H}_{m,m}$ is sequentially Cohen--Macaulay. Since sequential Cohen--Macaulayness is preserved under contractions (see \cite[Proposition~2.3]{ProvLouis} and \cite[Lemma~2.2(1)]{Russ11}), the hypergraph
\[
\mathcal{H}_{m,m} / x_{1,1,1} / x_{2,2,1} / \cdots / x_{m-3,m-3,1} / x_{m-2,m-2,2} / x_{m-1,m-1,2} / x_{m,m,2} / x_{m,m,1} \]
is sequentially Cohen--Macaulay. However, this contraction is isomorphic to the $4$-cycle $C_4$, which is not sequentially Cohen--Macaulay by \cite[Theorem~10]{Wood2}, a contradiction. Hence $\mathcal{H}_{m,m}$ is not sequentially Cohen--Macaulay.

It follows that $J(\mathcal{H}_{m,m})$ is not componentwise linear. Since polarization preserves componentwise linearity, we conclude that $J_{m,m}^2$ is not componentwise linear. As $J_{m,m}^2$ is generated in a single degree, it does not have a linear resolution. Finally, since an ideal with linear quotients is componentwise linear, $J_{m,m}^2$ does not have linear quotients.
\end{proof}

Theorems~\ref{2-ord} and \ref{equal-not} show that $J_{2,n}^q$ has linear quotients for every $q \ge 1$, whereas $J_{m,m}^2$ does not have linear quotients for any $m \ge 2$. Computations using the \textsc{SimplicialDecomposability} package \cite{Cook} in \textsc{Macaulay2} \cite{M2} further show that $J_{3,4}^2$ has linear quotients. On the other hand, $\reg(J_{4,5}^2) = 33$, while every minimal generator of $J_{4,5}^2$ has degree $16$. Thus, $J_{4,5}^2$ does not have a linear resolution and, consequently, does not have linear quotients. These examples indicate that the linear-quotient property of ordinary powers depends more subtly on $(m,n)$ than does that of symbolic powers.

This leads to the following question.

\begin{question}\label{ord-que}
What are the necessary and sufficient conditions on $(m,n)$ for $J_{m,n}^q$ to have linear quotients for every $q \ge 2$?
\end{question}

\section{Castelnuovo-Mumford regularity of $J_{m,n}$ and its symbolic powers}\label{reg-sec}

In this section, we study the Castelnuovo-Mumford regularity and the maximal degrees of minimal generators of the symbolic powers of $J_{m,n}$. We first recall the regularity of $J_{m,n}$, which exhibits a transition at the threshold $n=2m-1$. We then determine the maximal degree of a minimal generator of $J_{m,n}^{(q)}$ and show that it is linear in $q$. Finally, we determine the regularity of the higher symbolic powers. Unlike the case $q=1$, their behavior is governed by the distinction between the off-diagonal case $m<n$ and the diagonal case $m=n$.

We first record some consequences of the known results on the depth and Cohen-Macaulayness of chessboard complexes; see \cite[Corollary~1.4]{BLVZ94} and Theorem~\ref{ER-Terai}. 

\begin{proposition}\label{reg-base}
We have $\reg(J_{m,n}) = m(n-1)$ if $n \ge 2m-1$, whereas, if $n < 2m-1$, then
$m(n-1)+1 \le \reg(J_{m,n}) \le mn - \left\lfloor \frac{m+n+1}{3} \right\rfloor.$
\end{proposition}

We next study the maximal degree of a minimal generator of a symbolic power of $J_{m,n}$. In general, this degree need not be eventually linear in the symbolic exponent; see \cite[Theorem~5.15]{DHNT20}. For the ideals $J_{m,n}$, however, it has a particularly simple form: the maximal degree is linear in the symbolic exponent with zero constant term. We begin with the following setup.

\begin{setup}\label{set-up-deg}
Set $G = R_{m,n}^q$. For $1 \le r \le m$ and $1 \le s \le q$, let $H_{r,s}$ denote the induced subgraph of $G$ on
$\{x_{r,j,s} \mid 1 \le j \le n\}.$
Let $U$ be a maximal independent set of $G$ with $|U| < qm$.
\end{setup}

\begin{lemma}\label{1-lem}
Assume Setup~\ref{set-up-deg}. If $s \ge \left\lfloor \frac{q+1}{2} \right\rfloor + 1$, then $V(H_{r,s}) \cap U \neq \emptyset$ for every $1 \le r \le m$.
\end{lemma}

\begin{proof}
Suppose, to the contrary, that $V(H_{r,s}) \cap U = \emptyset$ for some $1 \le r \le m$. We first show that there exists $1 \le j \le n$ such that $U \cup \{x_{r,j,s}\}$ is independent. Suppose otherwise. Then, for each $j$, there exists a vertex of $U$ adjacent to $x_{r,j,s}$.

Fix $p \in \{1, \ldots, n\}$ and let $z \in U \cap N_G(x_{r,p,s})$. Suppose first that $z = x_{r,u,\ell}$. Since $z$ is adjacent to $x_{r,p,s}$, we have $u \ne p$ and $\ell + s \le q+1$. Since $U \cup \{x_{r,u,s}\}$ is also dependent, there exists $z' \in U \cap N_G(x_{r,u,s})$. If $z' = x_{r,v,w}$, then $u \ne v$ and $w+s \le q+1$, so $\ell + w \le q+1$, contradicting the independence of $U$ since $z\in N_G(z')$. Thus $z' = x_{r',u,s'}$ for some $r' \ne r$. Since $s+s' \le q+1$ and $s+\ell \le q+1$, we have $s'+\ell \le 2(q+1-s) \le q+1$. Hence $z$ and $z'$ are adjacent, again contradicting the independence of $U$.

Therefore, no vertex of $U \cap N_G(x_{r,p,s})$ lies in row $r$. Consequently, for each $p \in \{1, \ldots, n\}$, there exists a neighbor $x_{a,p,k} \in U$ with $a \ne r$ and $s+k \le q+1$. Since $n > m-1$, the pigeonhole principle yields distinct columns $p, p'$ and a row $a \ne r$ such that $x_{a,p,k}, x_{a,p',k'} \in U$. Moreover, $s+k \le q+1$ and $s+k' \le q+1$, and hence $k+k' \le 2(q+1-s) \le q+1$. Thus $x_{a,p,k}$ and $x_{a,p',k'}$ are adjacent, contradicting the independence of $U$.

Therefore, $U \cup \{x_{r,j,s}\}$ is independent for some $j \in \{1, \ldots, n\}$, contradicting the maximality of $U$. Hence $V(H_{r,s}) \cap U \neq \emptyset$.
\end{proof}

\begin{lemma}\label{q-odd}
Assume Setup~\ref{set-up-deg}. Suppose that $q$ is odd and set $t = \frac{q+1}{2}$. Then $V(H_{r,t}) \cap U \neq \emptyset$ for every $1 \le r \le m$.
\end{lemma}

\begin{proof}
Similarly, by an argument analogous to that in the proof of Lemma~\ref{1-lem}, we can show that if $V(H_{r,t}) \cap U =\emptyset$ then $U \cup \{x_{r,j,t}\}$ is independent for some $j.$ Hence it will contradict the maximality of U. So $V(H_{r,t}) \cap U \neq \emptyset$ for every $1 \le r \le m$.
\end{proof}

\begin{lemma}\label{main-lemma}
Assume Setup~\ref{set-up-deg}, and let $S_r := V(H_{r,q-\overline{t-1}})$. If $V(H_{r,t}) \cap U = \emptyset$, then $|S_r \cap U| \ge 2$.
\end{lemma}

\begin{proof}
Since $H_{r,t} \cap U = \emptyset$, we have $t \le \left\lfloor \frac{q+1}{2} \right\rfloor$. Now by Lemma~\ref{1-lem} $S_r\cap U\not= \varnothing$. Suppose, to the contrary, that $S_r \cap U = \{x_{r,1,q-\overline{t-1}}\}$. Then $\{x_{r,p,q-\overline{t-1}}\} \cup U$ is dependent for every $p \ne 1$.

\medskip
\noindent\textsc{Claim 1.} $x_{r,x_0,k'} \notin U$ for all $x_0 \ne 1$ and $1 \le k' < t$.

Indeed, suppose, to the contrary, that $x_{r,x_0,k'} \in U$ for some $x_0 \ne 1$ and $1 \le k' < t$. Then $x_{r,1,q-\overline{t-1}} \in N_G(x_{r,x_0,k'})$, which contradicts the fact that $U$ is independent. Hence $x_{r,x_0,k'} \notin U$ for all $x_0 \ne 1$ and $1 \le k' < t$.

\medskip
\noindent\textsc{Claim 2.} $x_{r,1,t-\gamma} \notin U$ for every $\gamma \in \mathbb{N}$.

Suppose, to the contrary, that $x_{r,1,t-\gamma} \in U$ for some $\gamma \in \mathbb{N}$. Since $H_{r,t} \cap U = \varnothing$, the set $U \cup \{x_{r,1,t}\}$ is dependent. Hence there exists $x_{s,z_1,z_2} \in U \cap N_G(x_{r,1,t})$. We consider two cases. If $s=r$, then $z_1 \ne 1$ and $z_2+t \le q+1$. Consequently, $z_2+(t-\gamma) \le q+1-\gamma < q+1$, and hence $x_{r,1,t-\gamma} \in N_G(x_{r,z_1,z_2})$, which contradicts the independence of $U$. If $s \ne r$, then necessarily $z_1=1$. Writing the vertex as $x_{s,1,y'}$, we have $y'+t \le q+1$. Thus $y'+(t-\gamma) \le q+1-\gamma < q+1$, and therefore $x_{r,1,t-\gamma} \in N_G(x_{s,1,y'})$, again contradicting the independence of $U$. Hence $x_{r,1,t-\gamma} \notin U$ for every $\gamma \in \mathbb{N}$.

Now, since $U \cup \{x_{r,1,t}\}$ is dependent, there exists $z \in N_G(x_{r,1,t}) \cap U$.

\medskip
\noindent\textsc{Claim 3.} The vertex $z$ is of the form $z = x_{r_0,1,p}$, $r_0 \ne r$, $p \le q-t+1$.

Indeed, write $z = x_{r',1,t'}$ with $r' \ne r$. Since $z \in N_G(x_{r,1,t})$, we have $t'+t \le q+1$, and hence $t' \le q-t+1$. It remains to show that $z$ cannot have the form $x_{r,s,z_0}$ with $s \ne 1$. Suppose, to the contrary, that $z = x_{r,s,z_0}$, $s \ne 1$. Since $z \in N_G(x_{r,1,t})$, we have $z_0+t \le q+1$, and hence $z_0 \le q-t+1$. On the other hand, for every $i = 2, \ldots, n$, $U \cup \{x_{r,i,q-\overline{t-1}}\}$ is dependent. In particular, for $i=s$, there exists $z' \in U \cap N_G(x_{r,s,q-\overline{t-1}})$. If $z' = x_{r,x',y'}$, then $x' \ne s$ and $y' < t$. Since $z_0 \le q-t+1$, it follows that $z = x_{r,s,z_0} \in N_G(x_{r,x',y'}) = N_G(z')$, contradicting the independence of $U$. Thus $z'$ must have the form $z' = x_{r'',s,w}$, $r'' \ne r$, $w \le t$. Again, $z = x_{r,s,z_0} \in N_G(z')$, which is a contradiction. Therefore $z \ne x_{r,s,z_0}$ for every $s \ne 1$. Consequently, $z$ must be of the form $z = x_{r_0,1,p}$, $r_0 \ne r$, $p \le q-t+1$.

Now, for every $i = 2, \ldots, n$, $U \cup \{x_{r,i,q-\overline{t-1}}\}$ is dependent.

\medskip
\noindent\textsc{Claim 4.} If $u \in N_G(x_{r,i,q-\overline{t-1}}) \cap U$, then $u = x_{r_1,i,p}$, $r_1 \ne r, r_0$, $p \le t$, $i \ne 1$.

Indeed, suppose that $u = x_{r,p',q'} \in U\cap N_G(x_{r,i,q-t+1})$. Then $q'<t$.Then by \textsc{Claim~1} $p'=1$. Now $q' \le t-1$. This contradicts \textsc{Claim~2}. Hence $u$ must have first coordinate different from $r$ and second coordinate different from $1$. Consequently, $u = x_{r_1,i,p}$ for some $r_1 \ne r$, $i \ne 1$, and $p \le t$.

It remains to show that $r_1 \ne r_0$. Suppose, to the contrary, that $r_1 = r_0$. By \textsc{Claim~3}, we have $z = x_{r_0,1,p_0}$ with $p_0 \le q-t+1$. Since $u = x_{r_0,i,p}$ with $p \le t$, the defining adjacency condition gives $p+p_0 \le q+1$. Hence $z \in N_G(u)$, contradicting the independence of $U$. Therefore $r_1 \ne r_0$, and the claim follows.

Finally, since $m-2 < n-1$, there exist two vertices $z_1, z_2 \in I$ with $z_1 \in N_G(z_2)$, contradicting the independence of $U$. Therefore, $|S_r \cap U| \ge 2$.
\end{proof}

The preceding lemmas provide the combinatorial ingredients needed to determine the degrees of the minimal generators of $J_{m,n}^{(q)}$. Combining them with the construction above, we obtain the following formula.

\begin{theorem}\label{deg-sym}
For every $q \ge 1$,
$\omega\!\left(J_{m,n}^{(q)}\right) = qm(n-1).$
\end{theorem}

\begin{proof}
Set $G = R_{m,n}^q$ and $S = \{x_{i,i,p} \mid 1 \le i \le m,\ 1 \le p \le q\}$. Then $S$ is an independent set of cardinality $qm$. Moreover, $S$ is maximal. Indeed, if $x_{a,b,c} \notin S$, then $a \ne b$, and hence $x_{a,b,c}$ is adjacent to $x_{a,a,1}$. Thus $S \cup \{x_{a,b,c}\}$ is not independent.

Now let $U$ be a maximal independent set of $G$. If $|U| < qm$, then Lemmas~\ref{1-lem}, \ref{q-odd}, and \ref{main-lemma} imply that $|U| \ge qm$, a contradiction. Hence every maximal independent set of $G$ has cardinality at least $qm$. So we obtain
$\beta_0(G) = qm,$
where $\beta_0(G)$ denotes the minimum cardinality of a maximal independent set of $G$. The complement of a maximal independent set is a minimal vertex cover. Therefore,
\[
\max\{|C| \mid C \text{ is a minimal vertex cover of } G\} = |V(G)| - \beta_0(G) = qmn - qm = qm(n-1).
\]
Consequently, $\omega(J(G)) = qm(n-1)$. By Lemma~\ref{lq-cl},
$\left(J_{m,n}^{(q)}\right)^{\pol} = J(G),$
and polarization preserves the degrees of minimal generators. Hence $\omega\!\left(J_{m,n}^{(q)}\right) = qm(n-1)$.
\end{proof}

We now turn to the regularity of the symbolic powers. The degree formula above provides a natural lower bound for the regularity. While the case $q=1$ is governed by the threshold $n=2m-1$, for higher symbolic powers the behavior is determined by whether $m<n$ or $m=n$. The next theorem gives the exact regularity in both cases.

\begin{theorem}\label{reg-sym}
For every integer $q \ge 2$,
\[
\reg\!\left(J_{m,n}^{(q)}\right) =
\begin{cases}
qm(n-1), & \text{if } m < n,\\[2mm]
qm(m-1)+1, & \text{if } m = n \ge 2.
\end{cases}
\]
\end{theorem}

\begin{proof}
 We first establish the lower bounds. Recall that $J_{m,n} = J(R_{m,n})$. Suppose that $m < n$. Since $\omega(J(R_{m,n})) = m(n-1)$, we have
\[
\reg\!\left(J(R_{m,n})^{(q)}\right) \ge q\,\omega\!\left(J(R_{m,n})\right) = qm(n-1).
\]
Now assume that $m = n \ge 2$ and set $G = R_{m,m}$. We claim that
$\reg\!\left(J(G)^{(q)}\right) \ge qm(m-1)+1.$
We prove this by induction on $m$. For $m=2$, we have $G = C_4 = K_{2,2}$, and \cite[Theorem~4.6]{ARSS} yields
$\reg\!\left(J(G)^{(q)}\right) = 2q+1 = qm(m-1)+1.$
Suppose that $m \ge 3$. By \cite[Lemma~4.4]{fakhari2025castelnuovo},
$\reg\!\left(J(G)^{(q)}\right) \ge \reg\!\left(J(G \setminus N_G[A])^{(q)}\right) + q|N_G(A)|.$
Taking $A = \{x_{1,1}\}$, we have $G \setminus N_G[A] \cong R_{m-1,m-1}$. Thus, by the induction hypothesis,
$\reg\!\left(J(G \setminus N_G[A])^{(q)}\right) \ge q(m-1)(m-2)+1.$
Since $|N_G(A)| = 2(m-1)$, it follows that
\[
\reg\!\left(J(G)^{(q)}\right) \ge q(m-1)(m-2)+1+2q(m-1) = qm(m-1)+1.
\]

We now prove the upper bounds simultaneously by induction on $m+q$. If $m=1$, then $G = R_{1,n} = K_n$ and
$J(G) = \left( \prod_{j \ne 1} x_j,\ \prod_{j \ne 2} x_j,\ \ldots,\ \prod_{j \ne n} x_j \right),$
which is the squarefree Veronese ideal of degree $n-1$. Hence $J(G)^{(q)} = J(G)^q$. Since every power of a squarefree Veronese ideal has a linear resolution,
$\reg\!\left(J(G)^{(q)}\right) = q(n-1) = qm(n-1).$
Thus, we may assume that $m \ge 2$. Set $G = R_{m,n}$ and $J = J(G)$. By \cite[Lemma~4.1]{fakhari2025castelnuovo}, it suffices to verify that, for every pair of subsets $A, B \subseteq V(G)$ satisfying $A \cap B = \emptyset$ and $A \cup B = V(G)$, where $(A)$ denotes the ideal generated by the variables corresponding to the vertices in $A$ and $x_B = \prod_{x_i \in B} x_i$, one has
\begin{equation}\label{req-eq}
\reg\!\left((J^{(q)} + (A)) : x_B\right) + |B| \le
\begin{cases}
qm(n-1), & \text{if } m < n,\\[2mm]
qm(m-1)+1, & \text{if } m = n.
\end{cases}
\end{equation}
We distinguish the following cases.

\medskip
\noindent\textbf{Case 1: $B = \emptyset$.}
Then $A = V(G)$, and hence
$(J^{(q)} + (A)) : x_B = (A).$
Since $(A)$ is generated by variables, we have
$\reg\!\left((J^{(q)} + (A)) : x_B\right) = 1.$
Therefore,
$\reg\!\left((J^{(q)} + (A)) : x_B\right) + |B| = 1,$
and \eqref{req-eq} follows.

\medskip
\noindent\textbf{Case 2: $A = \emptyset$.}
Then $B = V(G)$, and
$(J^{(q)} + (A)) : x_B = J^{(q)} : x_{V(G)}.$
By \cite[Lemma~3.4]{fakhari2017depth},
$J^{(q)} : x_{V(G)} = J^{(q-2)}.$
If $q=2$, then $J^{(0)} = S$, and hence
$\reg\!\left(J^{(2)} : x_{V(G)}\right) + |V(G)| = mn.$
If $m<n$, then $mn \le 2m(n-1)$, while, if $m=n$, then $m^2 \le 2m(m-1)+1$. Thus, \eqref{req-eq} holds for $q=2$.
Suppose that $q=3$. Then $J^{(3)} : x_{V(G)} = J$. By Proposition~\ref{reg-base}, if $n<2m-1,$ then,
$\reg(J_{m,n}) \le mn - \left\lfloor \frac{m+n+1}{3} \right\rfloor,$
we obtain
$\reg\!\left(J^{(3)} : x_{V(G)}\right) + |V(G)| \le 2mn - \left\lfloor \frac{m+n+1}{3} \right\rfloor.$

If $m<n<2m-1$, then
\[
3m(n-1) - \left( 2mn - \left\lfloor \frac{m+n+1}{3} \right\rfloor \right) = m(n-3) + \left\lfloor \frac{m+n+1}{3} \right\rfloor \ge 0.
\]
Also if $n\ge 2m-1,$ then
$\reg(J_{m,n})=m(n-1)$ so in that case,
\[3m(n-1)-m(n-1)-mn=2m(n-1)-mn=m(n-2)\ge0\]
If $m=n$, then
\[
3m(m-1)+1 - \left( 2m^2 - \left\lfloor \frac{2m+1}{3} \right\rfloor \right) = m(m-3)+1 + \left\lfloor \frac{2m+1}{3} \right\rfloor \ge 0.
\]
Thus, \eqref{req-eq} holds for $q=3$.
Finally, let $q \ge 4$. By the induction hypothesis,
\[
\reg\!\left(J^{(q-2)}\right) \le
\begin{cases}
(q-2)m(n-1), & m < n,\\[2mm]
(q-2)m(m-1)+1, & m = n.
\end{cases}
\]
Hence, if $m<n$,
$\reg\!\left(J^{(q)} : x_{V(G)}\right) + |V(G)| \le (q-2)m(n-1) + mn \le qm(n-1),$
whereas, if $m=n$,
$\reg\!\left(J^{(q)} : x_{V(G)}\right) + |V(G)| \le (q-2)m(m-1)+1+m^2 \le qm(m-1)+1.$
Thus, \eqref{req-eq} holds when $A = \emptyset$.

\medskip
\noindent\textbf{Case 3:} $A$ contains two adjacent vertices.

Suppose that $x_i, x_j \in A$ and $\{x_i, x_j\} \in E(G)$. Then every vertex cover of $G$ contains at least one of $x_i$ and $x_j$, and hence $J \subseteq (A)$. It follows that $J^{(q)} \subseteq (A)$. Since $A \cap B = \emptyset$, we have
$(J^{(q)} + (A)) : x_B = (A) : x_B = (A).$
Consequently,
$\reg\!\left((J^{(q)} + (A)) : x_B\right) = 1.$
Moreover, $|A| \ge 2$, and hence $|B| = |V(G)| - |A| = mn - |A| \le mn - 2$. Therefore,
$\reg\!\left((J^{(q)} + (A)) : x_B\right) + |B| \le mn - 1,$
which is bounded above by the right-hand side of \eqref{req-eq}.

\medskip
\noindent\textbf{Case 4:} $A$ is an independent set.

Since $A \ne \emptyset$, put $a = |A|$. Because $A$ is independent in the rook graph, its vertices lie in distinct rows and distinct columns. Hence $1 \le a \le m$. The $a$ vertices of $A$ determine $a$ rows and $a$ columns. Therefore,
$|N_G(A)| = a(m+n-a-1).$
Let $u = \prod_{x_i \in N_G(A)} x_i$. Then $\deg(u) = a(m+n-a-1)$. Moreover, deleting the closed neighborhood of $A$ removes precisely those $a$ rows and $a$ columns, and hence
$G \setminus N_G[A] \cong R_{m-a,n-a}.$
By \cite[Lemma~4.2]{fakhari2025castelnuovo},
$J^{(q)} + (A) = u^q J(G \setminus N_G[A])^{(q)} + (A).$
Since $A \cap B = \emptyset$,
\[
(J^{(q)} + (A)) : x_B = \left( u^q J(G \setminus N_G[A])^{(q)} : x_B \right) + (A).
\]
Furthermore, $B = N_G(A) \cup (V(G) \setminus N_G[A])$, and hence $x_B = u\, x_{V(G) \setminus N_G[A]}$. Therefore, by \cite[Lemma~3.4]{fakhari2017depth},
$(J^{(q)} + (A)) : x_B = u^{q-1} J(G \setminus N_G[A])^{(q-2)} + (A).$
Since the elements of $A$ form a regular sequence modulo $u^{q-1} J(G \setminus N_G[A])^{(q-2)}$, we obtain
\begin{equation}\label{reg-eq2}
\reg\!\left((J^{(q)} + (A)) : x_B\right) = (q-1)\deg(u) + \reg\!\left(J(G \setminus N_G[A])^{(q-2)}\right).
\end{equation}

We now distinguish the cases $q=2$, $q=3$, and $q \ge 4$.

\medskip
\noindent\textit{Subcase 4.1: $q=2$.}

Since $J(G \setminus N_G[A])^{(0)} = S$, equation~\eqref{reg-eq2} yields
$\reg\!\left((J^{(2)} + (A)) : x_B\right) = a(m+n-a-1).$
Moreover, $|B| = mn-a$. Consequently,
$\reg\!\left((J^{(2)} + (A)) : x_B\right) + |B| = a(m+n-a-1) + mn-a.$
If $m<n$, then
\[
2m(n-1) - \left[ a(m+n-a-1) + mn-a \right] = (m-a)(n-a-2) \ge 0.
\]
Indeed, putting $a=m$ gives zero, and since $1 \le a < m < n$, we have $m-a \ge 0$ and $n-a-2 \ge 0$. Thus,
$\reg\!\left((J^{(2)} + (A)) : x_B\right) + |B| \le 2m(n-1).$
If $m=n$, then
\[
2m(m-1)+1 - \left[ a(2m-a-1) + m^2-a \right] = (m-a-1)^2 \ge 0.
\]
Hence,
$\reg\!\left((J^{(2)} + (A)) : x_B\right) + |B| \le 2m(m-1)+1.$

\medskip
\noindent\textit{Subcase 4.2:} $q=3$ and $m<n$.

Put $r = m-a$ and $s = n-a$. Then $r < s$, and $G \setminus N_G[A] \cong R_{r,s}$. Suppose $n < 2m-a-1$. Then by Proposition~\ref{reg-base},
$\reg(J_{r,s}) \le rs - \left\lfloor \frac{r+s+1}{3} \right\rfloor,$
and equation~\eqref{reg-eq2} gives
\[
\begin{aligned}
\reg\!\left((J^{(3)} + (A)) : x_B\right) + |B| &\le 2a(m+n-a-1) + (m-a)(n-a) \\
&\quad - \left\lfloor \frac{m+n-2a+1}{3} \right\rfloor + mn - a.
\end{aligned}
\]
Moreover,
\[
\begin{aligned}
&3m(n-1) - \Bigg[ 2a(m+n-a-1) + (m-a)(n-a) 
- \left\lfloor \frac{m+n-2a+1}{3} \right\rfloor + mn - a \Bigg] \\
&= (m-a)(n-a-3) + \left\lfloor \frac{m+n-2a+1}{3} \right\rfloor.
\end{aligned}
\]
Since $1 \le a \le m < n$, the right-hand side is nonnegative. Indeed, putting $r = m-a$ and $d = n-m$, we have $r \ge 0$, $d \ge 1$, and
\[
(m-a)(n-a-3) + \left\lfloor \frac{m+n-2a+1}{3} \right\rfloor = r(r+d-3) + \left\lfloor \frac{2r+d+1}{3} \right\rfloor.
\]
If $r+d \ge 3$, the first term is nonnegative. If $r+d = 2$, then either $r=0$ and $d=2$, or $r = d = 1$. The above expression equals $-1 + \left\lfloor \frac{4}{3} \right\rfloor = 0$ when $r=d=1$, and equals $1$ when $r=0$ and $d=2$. If $r+d=1$, then the only possibility is $r=0$ and $d=1$, and the above expression is again nonnegative.
Thus,
$\reg\!\left((J^{(3)} + (A)) : x_B\right) + |B| \le 3m(n-1).$
Now suppose $n \ge 2m-a-1$. Then
$\reg(J_{r,s}) = (m-a)(n-a-1),$
and similarly it can be shown that
$3m(n-1) - 2a(m+n-a-1) -(m-a)(n-a-1) - (mn-a) = (m-a)(n-a-2) \ge 0.$
Therefore,
$\reg\!\left((J^{(3)} + (A)) : x_B\right) + |B| \le 3m(n-1).$

\medskip
\noindent\textit{Subcase 4.3:} $q=3$ and $m=n$.

If $a=m$, then $G \setminus N_G[A]$ is empty. Hence
$\reg\!\left((J^{(3)} + (A)) : x_B\right) = 2m(m-1).$
Moreover, $|B| = m^2-m = m(m-1)$. Consequently,
\[
\reg\!\left((J^{(3)} + (A)) : x_B\right) + |B| = 3m(m-1) < 3m(m-1)+1.
\]
Now suppose that $1 \le a \le m-1$, and put $r = m-a$. Then $G \setminus N_G[A] \cong R_{r,r}$. By Proposition~\ref{reg-base},
$\reg(J_{r,r}) \le r^2 - \left\lfloor \frac{2r+1}{3} \right\rfloor,$
together with \eqref{reg-eq2}, we obtain
\[
\reg\!\left((J^{(3)} + (A)) : x_B\right) + |B| \le 2a(2m-a-1) + r^2 - \left\lfloor \frac{2r+1}{3} \right\rfloor + m^2 - a.
\]
Since $r = m-a$, a direct calculation yields
\[
\begin{aligned}
&3m(m-1)+1 - \Bigg[ 2a(2m-a-1) + r^2 
- \left\lfloor \frac{2r+1}{3} \right\rfloor + m^2 - a \Bigg] \\
&= r^2-3r+1 + \left\lfloor \frac{2r+1}{3} \right\rfloor.
\end{aligned}
\]
For $r=1,2$, the right-hand side is $0$, while for $r \ge 3$ it is positive. Hence
\[
\reg\!\left((J^{(3)} + (A)) : x_B\right) + |B| \le 3m(m-1)+1.
\]

\medskip
\noindent\textit{Subcase 4.4:} $q \ge 4$ and $m<n$.

If $a=m$, then $G \setminus N_G[A]$ is empty. Hence, by \eqref{reg-eq2},
$\reg\!\left((J^{(q)} + (A)) : x_B\right) = (q-1)m(n-1).$
Moreover, $|B| = mn-m = m(n-1)$. Consequently,
$\reg\!\left((J^{(q)} + (A)) : x_B\right) + |B| = qm(n-1).$
Now assume that $1 \le a \le m-1$. By the induction hypothesis,
$$\reg\!\left(J(R_{m-a,n-a})^{(q-2)}\right) \le (q-2)(m-a)(n-a-1).$$
Therefore, by \eqref{reg-eq2},
\[
\begin{aligned}
&\reg\!\left((J^{(q)} + (A)) : x_B\right) + |B| \\
&\le (q-1)a(m+n-a-1) + (q-2)(m-a)(n-a-1) + mn-a.
\end{aligned}
\]
A direct calculation gives
\[
\begin{aligned}
&qm(n-1) - \Big[ (q-1)a(m+n-a-1) + (q-2)(m-a)(n-a-1) + mn-a \Big] \\
&= (m-a)(n-a-2) \ge 0.
\end{aligned}
\]
Indeed, since $1 \le a \le m-1$ and $m<n$, we have $m-a \ge 1$ and $n-a-2 \ge n-m-1 \ge 0$. Thus
$\reg\!\left((J^{(q)} + (A)) : x_B\right) + |B| \le qm(n-1).$

\medskip
\noindent\textit{Subcase 4.5:} $q \ge 4$ and $m=n$.

If $a=m$, then $G \setminus N_G[A]$ is empty. Hence, by \eqref{reg-eq2},
$\reg\!\left((J^{(q)} + (A)) : x_B\right) = (q-1)m(m-1).$
Since $|B| = m^2-m = m(m-1)$, we obtain
\[
\reg\!\left((J^{(q)} + (A)) : x_B\right) + |B| = qm(m-1) < qm(m-1)+1.
\]
Now suppose that $1 \le a \le m-1$, and put $r = m-a$. By the induction hypothesis,
$\reg\!\left(J(R_{r,r})^{(q-2)}\right) \le (q-2)r(r-1)+1.$
Hence
\[
\reg\!\left((J^{(q)} + (A)) : x_B\right) + |B| \le (q-1)a(2m-a-1) + (q-2)r(r-1) + 1 + m^2 - a.
\]
Since $r = m-a$, a direct calculation gives
\[
\begin{aligned}
&qm(m-1)+1 - \Big[ (q-1)a(2m-a-1)
+ (q-2)r(r-1) + 1 + m^2 - a \Big] 
= r(r-2).
\end{aligned}
\]
If $r \ge 2$, then $r(r-2) \ge 0$, and hence
$\reg\!\left((J^{(q)} + (A)) : x_B\right) + |B| \le qm(m-1)+1.$
It remains to consider the case $r=1$, that is, $a=m-1$. Then $G \setminus N_G[A] \cong R_{1,1}$, whose cover ideal is the whole polynomial ring. Thus
$\reg\!\left(J(R_{1,1})^{(q-2)}\right) = 0.$
It follows from \eqref{reg-eq2} that
$\reg\!\left((J^{(q)} + (A)) : x_B\right) = (q-1)m(m-1).$
Moreover, $|B| = m^2-(m-1) = m(m-1)+1$. Consequently,
$\reg\!\left((J^{(q)} + (A)) : x_B\right) + |B| = qm(m-1)+1.$
Thus, \eqref{req-eq} holds in all cases. By \cite[Lemma~4.1]{fakhari2025castelnuovo},
\[
\reg\!\left(J_{m,n}^{(q)}\right) \le
\begin{cases}
qm(n-1), & m < n,\\[2mm]
qm(m-1)+1, & m = n.
\end{cases}
\]
This completes the proof.
\end{proof}

\begin{remark}
The case $m < n < 2m-1$ exhibits an interesting contrast between linear quotients and regularity. By Theorem~\ref{main-classification}, the symbolic powers $J_{m,n}^{(q)}$ fail to have linear quotients for odd $q$. Nevertheless, Theorem~\ref{reg-sym} shows that
\[
\reg\!\left(J_{m,n}^{(q)}\right) = \omega\!\left(J_{m,n}^{(q)}\right) = qm(n-1)
\]
for every $q \ge 2$. Thus, the failure of linear quotients for odd symbolic powers does not affect the regularity: it is still determined by the maximal degree of a minimal generator.
\end{remark}

As a consequence of Theorems~\ref{2-ord} and~\ref{reg-sym}, we obtain the following.

\begin{corollary}\label{reg-comparision}
For every $q \ge 1$ and $n \ge 2$,
$\reg\!\left(J_{2,n}^q\right) = \reg\!\left(J_{2,n}^{(q)}\right).$
Moreover, if $m<n$, then
$\reg\!\left(J_{m,n}^{(q)}\right) \le \reg\!\left(J_{m,n}^q\right)$
for all $q \ge 2$. If $m=n$, then
$\reg\!\left(J_{m,m}^{(2)}\right) \le \reg\!\left(J_{m,m}^2\right).$
\end{corollary}

This naturally leads to the following question.

\begin{question}
If $m=n$, does
$\reg\!\left(J_{m,m}^{(q)}\right) \le \reg\!\left(J_{m,m}^q\right)$
hold for every $q \ge 3$? More generally, can one characterize all triples $(m,n,q)$ with $m \le n$ and $q \ge 2$ for which
$\reg\!\left(J_{m,n}^{(q)}\right) = \reg\!\left(J_{m,n}^q\right)?$
\end{question}

\section{Initial degree, Waldschmidt constant, and $\vv$-number}\label{vnumber-sec}

In this section, we study the initial degree and the $\vv$-number of the symbolic and ordinary powers of $J_{m,n}$. We begin by determining the initial degrees of the symbolic powers and, consequently, the Waldschmidt constant. The parity of the symbolic exponent plays a crucial role in the resulting formula.

\begin{theorem}\label{min-deg}
For every $q \ge 1$,
\[
\alpha\!\left(J_{m,n}^{(q)}\right) =
\begin{cases}
\dfrac{mnq}{2}, & \text{if } q \text{ is even},\\[2mm]
\dfrac{mn(q-1)}{2} + m(n-1), & \text{if } q \text{ is odd}.
\end{cases}
\]
Consequently,
$\widehat{\alpha}(J_{m,n}) = \frac{mn}{2}.$
\end{theorem}

\begin{proof}
By Lemma~\ref{lq-cl}, the initial degree of $J_{m,n}^{(2)}$ equals the minimum cardinality of a minimal vertex cover of $R_{m,n}^2$.

First, we show that there exists a minimal vertex cover of cardinality $mn$. Consider
$S = \left\{ x_{i,j,1} \mid 1 \le i \le m,\ 1 \le j \le n \right\}.$
Suppose that $S' \subset S$ is a vertex cover. If $x_{i,j,1} \in S \setminus S'$, then $x_{i,j',2} \in S'$ for every $1 \le j' \le n$ with $j' \ne j$, which is a contradiction. Therefore, $S$ is a minimal vertex cover, and $|S| = mn$.
For $1 \le r \le m$ and $1 \le s \le 2$, let $H_{r,s}$ denote the induced subgraph of $R_{m,n}^{2}$ on the vertex set $\{x_{r,j,s} \mid 1 \le j \le n\}$. Suppose, for contradiction, that there exists a minimal vertex cover $S'$ of $R_{m,n}^{2}$ such that $|S'| < mn$. Since $S'$ is a vertex cover, we have
$|H_{r,1} \cap S'| \ge n-1$
for all $1 \le r \le m$. Suppose that there exists $r$ such that $|H_{r,1} \cap S'| = n-1$. Without loss of generality, let $x_{r,t,1} \notin S'$. Then $|H_{r,2} \cap S'| \ge n-1$. Hence, $|S'| \ge mn$, a contradiction. Therefore, the minimum cardinality of a minimal vertex cover of $R_{m,n}^{2}$ is $mn$. Hence,
$\alpha\!\left(J(R_{m,n})^{(2)}\right) = mn.$
Also observe that if $S$ is a minimal vertex cover of $R_{m,n}^1$ then $|H_{r,1}\cap S|\ge n-1$ for all $r=1,\dots,m$ and there exist a minimal vertex cover of cardinality $m(n-1)$ of $R_{m,n}^1$. So $\alpha(J_{m,n}) = m(n-1).$ Now by \cite[Theorem~5.1]{HHT07}, for $q \ge 1$,
$J_{m,n}^{(2q')} = \left(J_{m,n}^{(2)}\right)^{q'}$
and
$J_{m,n}^{(2q'+1)} = J_{m,n}\left(J_{m,n}^{(2)}\right)^{q'}.$
Since $\alpha(J_{m,n}) = m(n-1)$ and $\alpha(J_{m,n}^{(2)}) = mn$, the asserted formula follows for every $q \ge 1$.
\end{proof}

\begin{remark}
It is known that the resurgence of a homogeneous ideal $I$ satisfies the Bocci-Harbourne bound \cite{BH10}
$\frac{\alpha(I)}{\widehat{\alpha}(I)} \le \rho(I),$
where
\[
\rho(I) = \sup\left\{ \frac{s}{t} \mid I^{(s)} \not\subseteq I^t \right\}.
\]
This naturally leads to the question posed by Jayanthan, Kumar, and Mukundan \cite[Question~4.13]{JKM22}: classify the graphs $G$ for which
$\rho(J(G)) = \frac{\alpha(J(G))}{\widehat{\alpha}(J(G))}.$
The rook graphs provide a natural family satisfying this equality. Indeed, $R_{m,n}$ is a perfect graph and, hence, by \cite[Corollary~4.6]{JKM22},
$\rho(J_{m,n}) = 2-\frac{2}{n} = \frac{2(n-1)}{n}.$
On the other hand, by Theorem~\ref{min-deg}, $\alpha(J_{m,n}) = m(n-1)$ and $\widehat{\alpha}(J_{m,n}) = \frac{mn}{2}$, and consequently
\[
\frac{\alpha(J_{m,n})}{\widehat{\alpha}(J_{m,n})} = \frac{2(n-1)}{n} = \rho(J_{m,n}).
\]
Thus, the rook graphs form an infinite family for which the Bocci-Harbourne bound is attained.
\end{remark}

The following theorem determines the $\vv$-number of the symbolic powers of $J_{m,n}$. The formula depends on the parity of $q$ and, in the square case, requires a separate expression.

\begin{theorem}\label{asyv}
\begin{enumerate}
\item If $m \ne n$, then
\[
\vv\!\left(J_{m,n}^{(q)}\right) =
\begin{cases}
m(n-1)-1+\dfrac{q-1}{2}mn, & \text{if } q \text{ is odd},\\[2mm]
m+n-4+\dfrac{q}{2}mn, & \text{if } q \text{ is even}.
\end{cases}
\]

\item If $m=n$, then
\[
\vv\!\left(J_{m,n}^{(q)}\right) =
\begin{cases}
\dfrac{m^2q+m^2-2m}{2}, & \text{if } q \text{ is odd},\\[2mm]
\dfrac{q}{2}m^2+2m-4, & \text{if } q \text{ is even}.
\end{cases}
\]
\end{enumerate}
\end{theorem}

\begin{proof}
We use the following characterization of the $\vv$-number for cover ideals \cite[Remark~2.5]{Sa24}. Let $\mathcal{C}_G$ denote the collection of vertex covers of $G$. Then
\[
\vv(J(G)) = \min\left\{ |C| : C \notin \mathcal{C}_G,\ C \cup \{x_i\},\ C \cup \{x_j\} \in \mathcal{C}_G \text{ for some } \{x_i,x_j\} \in E(G) \right\}.
\]
Moreover, in this case,
$(J(G) : x_C) = (x_i, x_j).$

\medskip
\noindent\textbf{Case 1:} $q$ is odd.

\medskip
\noindent\textit{Subcase 1.1:} $m \ne n$.

Set $G = R_{m,n}^q$. We first show that there exists an induced subgraph of $G$ with exactly one edge and cardinality $\frac{q-1}{2}mn+m+1$. Consider the induced subgraph $G[S]$, where
\[
S = \left\{ x_{i,j,k} \mid 1 \le i \le m,\ 1 \le j \le n,\ \tfrac{q+3}{2} \le k \le q \right\} \cup \left\{ x_{r,r+1,\frac{q+1}{2}} \mid 1 \le r \le m \right\} \cup \left\{ x_{1,1,\frac{q+1}{2}} \right\}.
\]
Then $|E(G[S])| = 1$ and
$|S| = \left( \frac{q-1}{2} \right) mn + m + 1.$
In particular, consider the first set $\emptyset$ when $q=1$. We show that this is maximal. Suppose, to the contrary, that there exists an induced subgraph $G[S']$ such that $|E(G[S'])| = 1$ and
$|S'| > \left( \frac{q-1}{2} \right) mn + m + 1.$
Let $\{x_{i,j,k}, x_{p,q',r}\}$ be the unique edge of $G[S']$. For $1 \le i_0 \le m$, let $G(i_0)$ denote the induced subgraph of $G$ on
$\{x_{i_0,j,k} \mid 1 \le j \le n,\ 1 \le k \le q\}.$

We first prove the following claim.

\medskip
\noindent\emph{Claim.} Suppose $x_{i',j',k'} \in S'$ with $1 \le k' \le \frac{q-1}{2}$ and $H_{i',t}\cap S'=\varnothing$ for all $1\le t <k'$ and $G(i')$ does not contain the unique edge of $S'$, then
$|G(i') \cap S'| \le \left( \frac{q-3}{2} \right)n + 3.$

\begin{proof}[Proof of the claim]
Since $x_{i',j',k'} \in S'$ and $1 \le k' \le \frac{q-1}{2}$, the condition $|E(G[S'])| = 1$ implies
\[
|G(i') \cap S'| \le \bigl( (q+1-k') - k' + 1 \bigr) + (k'-1)n.
\]
Consequently,
\[
\begin{aligned}
|G(i') \cap S'| &\le q - 2(k'-1) + (k'-1)n 
= (k'-1)(n-2) + q \\
&\le \left( \frac{q-3}{2} \right)(n-2) + q 
= \left( \frac{q-3}{2} \right)n + 3.
\end{aligned}
\]
This proves the claim.
\end{proof}
Suppose $x_{i',j',k'} \in S'$. Observe that if $H_{i',t} \cap S' = \varnothing$ either for all $1 \le t < \frac{q+1}{2}$, or for $1 \le t \le \frac{q+1}{2}$ and $G(i')$ does not contain the unique edge of $S'$, then in both cases $|G(i') \cap S'| \le \left(\frac{q-1}{2}\right)n + 1$.
Also observe that, 
$\left( \frac{q-3}{2} \right)n + 3 \le \left( \frac{q-1}{2} \right)n + 1$.

Now suppose that $i \ne p$. Then,
$|G(i) \cap S'| \le \left( \frac{q-1}{2} \right)n + 1$
for every $1 \le i \le m$, which gives the contradiction $|S'|\le \left( \frac{q-1}{2} \right) mn + m $. Therefore, $i = p$. \\
It follows that
$|G(i) \cap S'| \le \left( \frac{q-1}{2} \right)n + 2,$
whereas
$|G(j) \cap S'| \le \left( \frac{q-1}{2} \right)n + 1$
for all $j \ne i$. Consequently,
$|S'| \le \left( \frac{q-1}{2} \right)mn + m + 1,$
again a contradiction. Therefore,
\[
\max\Bigl\{ |S| : S \subseteq V(G),\ |E(G[S])| = 1 \Bigr\} = \left( \frac{q-1}{2} \right)mn + m + 1.
\]
By the characterization of the $\vv$-number,
\[
\begin{aligned}
\vv\!\left(J_{m,n}^{(q)}\right) &= qmn - \left( \left( \frac{q-1}{2} \right)mn + m + 1 \right) 
= m(n-1)-1+\frac{q-1}{2}mn.
\end{aligned}
\]
\smallskip

\medskip
\noindent\textit{Subcase 1.2: $m=n$.}

We first show that there exists an induced subgraph of $G$ with exactly one edge and cardinality $\left( \frac{q-1}{2} \right)m^2+m$. Consider the induced subgraph $G[S]$, where
\[
\begin{aligned}
S ={}& \left\{ x_{i,j,k} \mid 1 \le i,j \le m,\ \tfrac{q+3}{2} \le k \le q \right\} 
\cup \left\{ x_{1,1,\frac{q+1}{2}},\ x_{1,2,\frac{q+1}{2}},\ x_{2,3,\frac{q+1}{2}},\ \ldots,\ x_{m-1,m,\frac{q+1}{2}} \right\}.
\end{aligned}
\]
Then $|E(G[S])| = 1$ and
$|S| = \left( \frac{q-1}{2} \right)m^2 + m.$
We show that this is maximal. Suppose, to the contrary, that there exists an induced subgraph $G[S']$ such that $|E(G[S'])| = 1$ and
$|S'| > \left( \frac{q-1}{2} \right)m^2 + m.$
Let $\{x_{i,j,k}, x_{p,q',r}\}$ be the unique edge of $G[S']$. For $1 \le i_0 \le m$, let $G(i_0)$ denote the induced subgraph of $G$ on
$\{x_{i_0,j,k} \mid 1 \le j \le m,\ 1 \le k \le q\}.$
The following claim is identical to that in \textit{Subcase~1.1.}
\medskip

\noindent\emph{Claim.}
Suppose $x_{i',j',k'}\in S'$ with
$1\le k'\le\frac{q-1}{2}$ and $H_{i',t}\cap S'=\varnothing$ for all $1\le t < k'$ and $G(i')$ does not contain the unique edge of $S'$, then
$|G(i')\cap S'|
\le
\left(\frac{q-3}{2}\right)m+3
\le
\left(\frac{q-1}{2}\right)m+1.$

\medskip

Suppose that $i \ne p$. Then
$|G(i) \cap S'| \le \left( \frac{q-1}{2} \right)m + 1$
for every $1 \le i \le m$, which gives a contradiction that $|S'|> \left( \frac{q-1}{2} \right)m^2 + m$. Therefore, $i = p$. It follows that
$|G(i) \cap S'| \le \left( \frac{q-1}{2} \right)m + 2.$
If $m\ge 3$ there exists a row $j \ne i$ such that
$|G(j) \cap S'| \le \left( \frac{q-1}{2} \right)m.$
Consequently,
$|S'| \le \left( \frac{q-1}{2} \right)m^2 + m,$
again a contradiction. If $m=2$ and if $|G(i)\cap S'| \le \left( \frac{q-1}{2} \right)m + 1,$ then also we get the same contradiction. Now let $m=2$ and $|G(i)\cap S'| = \left( \frac{q-1}{2} \right)m + 2,$ then $|G(j)\cap S'| \leq \left( \frac{q-1}{2} \right)m$ for $i\not= j$ so again we get the same contradiction. Therefore,
\[
\max\Bigl\{ |S| : S \subseteq V(G),\ |E(G[S])| = 1 \Bigr\} = \left( \frac{q-1}{2} \right)m^2 + m.
\]
By the characterization of the $\vv$-number,
\[
\begin{aligned}
\vv\!\left(J_{m,n}^{(q)}\right) &= qm^2 - \left( \left( \frac{q-1}{2} \right)m^2 + m \right) 
= \frac{m^2q+m^2-2m}{2}.
\end{aligned}
\]

\medskip
\noindent\textbf{Case 2:} $q$ is even.

We first show that there exists an induced subgraph of $G$ with exactly one edge and cardinality $\frac{q}{2}mn-m-n+4$. Consider the induced subgraph $G[S]$, where
\[
\begin{aligned}
S ={}& \left\{ x_{i,j,k} \mid 1 \le i \le m,\ 1 \le j \le n,\ \tfrac{q+4}{2} \le k \le q \right\} 
\cup \left\{ x_{1,1,\frac{q}{2}},\ x_{1,1,\frac{q+2}{2}},\ x_{1,2,\frac{q+2}{2}} \right\} \\
&\cup \left\{ x_{r,s,\frac{q+2}{2}} \mid 2 \le r \le m,\ 2 \le s \le n \right\}.
\end{aligned}
\]
Then $|E(G[S])| = 1$ and
$|S| = \frac{q}{2}mn - m - n + 4.$
We show that this is maximal. Suppose, to the contrary, that there exists an induced subgraph $G[S']$ such that $|E(G[S'])| = 1$ and
$|S'| > \frac{q}{2}mn - m - n + 4.$
Let $\{x_{i,j,k}, x_{p,q',r}\}$ be the unique edge of $G[S']$. For $1 \le i_0 \le m$, let $G(i_0)$ denote the induced subgraph of $G$ on
$\{x_{i_0,j,k} \mid 1 \le j \le n,\ 1 \le k \le q\}.$

\medskip
\noindent\emph{Claim.} Suppose $x_{i',j',k'} \in S'$ with $k' \le \frac{q-2}{2}$ and $H_{i',t}\cap S'=\varnothing$ for all $1\le t <k'$ and $G(i')$ does not contain the unique edge of $S'$, then
$|G(i') \cap S'| \le \left( \frac{q-4}{2} \right)n + 4.$

\begin{proof}[Proof of the claim]
As in the proof of the claim in Subcase~1.1, the condition $|E(G[S'])| = 1$ gives
\[
|G(i') \cap S'| \le \{(q+1-k')-k'+1\} + \left( q-(q+1-k'+1)+1 \right)n = q + (k'-1)(n-2).
\]
Since $k' \le (q-2)/2$, we obtain
$|G(i') \cap S'| \le q + \left( \frac{q-4}{2} \right)(n-2) = \left( \frac{q-4}{2} \right)n + 4.$
This proves the claim.
\end{proof}

Now observe that if $n \ge 3$, then
$\left( \frac{q-4}{2} \right)n + 4 \le \frac{q}{2}n - 1 = \left( \frac{q-2}{2} \right)n + (n-1),$
and if $n \le 2$, that is, if $n=2$, then
$\left( \frac{q-4}{2} \right)n + 4 \le \frac{q}{2}n = \left( \frac{q-2}{2} \right)n + n.$

Let neither $r$ nor $k$ be equal to $\frac{q}{2}$. Then either $k$ or $r$ is at most $\frac{q-2}{2}$. Without loss of generality, let $k \le \frac{q-2}{2}$. Observe that for $i' \ne i$, if $x_{i',b,\frac{q}{2}} \in S' \cap G(i')$ and if $H_{i',u} \cap S' = \emptyset$ for all $1 \le u \le \frac{q}{2}-1$, then
\begin{align*}
    &|S' \cap G(i')| \le 2 + \left( (q+1-k) - \left( \frac{q}{2}+2 \right) + 1 \right)(n-1) + (k-1)n \\
    &= 2 + (k-1)n + \left( \frac{q}{2}-k \right)(n-1) \le \frac{q}{2}n - 1 = \left( \frac{q-2}{2} \right)n + (n-1).
\end{align*}
Also observe that for $i' \ne i$, if $x_{i',b,c} \in S' \cap G(i')$ and if $H_{j,u} \cap S' = \emptyset$ for all $1 \le u \le \frac{q}{2}$, then
$|S' \cap G(i')| \le (k-1)n + \left( \frac{q}{2}-k+1 \right)(n-1) \le \frac{q}{2}n - 1 = \left( \frac{q-2}{2} \right)n + (n-1).$
Now assume first $n \le 2$, that is, $n=2$. Then $m \le 2$. We have
$|S' \cap G(i)| \le \left( \frac{q-4}{2} \right)n + 4 \le \frac{q}{2}n.$
Also, since $n \le 2$, we have
$\left( \frac{q-2}{2} \right)n + (n-1) \le \left( \frac{q-2}{2} \right)n + 1 \le \frac{q}{2}n.$
So $|G(x) \cap S'| \le \frac{q}{2}n$ for all $x$. Hence

$|S'| \le \frac{qmn}{2} \le \frac{qmn}{2} - m - n + 4.$
Now assume $n \ge 3$. We have
$|G(i) \cap S'| \le \left( \frac{q-4}{2} \right)n + 4 \le \left( \frac{q-2}{2} \right)n + 3.$
Also, we have
$|G(j) \cap S'| \le \left( \frac{q-2}{2} \right)n + (n-1)$
for all $j \ne i$. Therefore,
$|S'| \le \left( \frac{q-2}{2} \right)mn + (m-1)(n-1) + 3 \le \frac{qmn}{2} - m - n + 4,$
which is a contradiction.
Hence at least one endpoint of the unique edge of $G[S']$ lies in the level $q/2$. Without loss of generality, let $k = \frac{q}{2}$. Now suppose that $p \ne i$ and $r \ne \frac{q+2}{2}$. Then $r \le \frac{q}{2}$. If $r < \frac{q}{2}$, then by the argument used previously, we obtain a contradiction. So let $r = \frac{q}{2}$. Without loss of generality, let $x_{i,j,\frac{q}{2}} = x_{1,1,\frac{q}{2}}$ and $x_{p,q',\frac{q}{2}} = x_{2,1,\frac{q}{2}}$. Then observe that $|G(i) \cap S'| \le \frac{q}{2}n - n+1$ for $i = 1, 2$, and if $n \ge 3$, then for $i\not=1,2$ $|G(i) \cap S'| \le \left(\frac{q-2}{2}\right)n + (n-1)$. So $|S'| \le \frac{qmn}{2} - m - n + 4$, which is a contradiction. Now let $n=2$ then $m=2$. Observe that $|G(i)\cap S'|\le q-1$ for $i=1,2$. Therefore $|S'|\le 2q-2< 2q=\frac{qmn}{2}-m-n+4$, which is a contradiction. So either $p = i$ or $r = \frac{q+2}{2}$. Now let $p \ne i$ and $r = \frac{q+2}{2}$. Without loss of generality, let $x_{i,j,\frac{q}{2}} = x_{1,2,\frac{q}{2}}$ and $x_{p,q',\frac{q+2}{2}} = x_{2,2,\frac{q+2}{2}}$. Now observe that $|G(1) \cap S'| \le \frac{q}{2}n + 2 - n$ and $|G(2) \cap S'| \le \frac{q}{2}n$. If $n \ge 3$, then for all $i \ne 1, 2$ we have $|G(i) \cap S'| \le \left(\frac{q-2}{2}\right)n + (n-1)$. Then $|S'| \le \frac{qmn}{2} - m - n + 4$, which is a contradiction. Now let $n=2$ then $m=2.$ Observe that for $i=1,2$ $|G(i)\cap S'|\le q$. So $|S'|\le 2q=\frac{qmn}{2}-m-n+4,$ which is a contradiction. Now let $p = i$ and $r \ne \frac{q+2}{2}$. Then $r \le \frac{q}{2}$. If $r < \frac{q}{2}$, then by the argument used previously, we obtain a contradiction. So let $r = \frac{q}{2}$. Without loss of generality, let $x_{i,j,\frac{q}{2}} = x_{1,1,\frac{q}{2}}$ and $x_{p,q',\frac{q}{2}} = x_{1,2,\frac{q}{2}}$. Then $|G(1) \cap S'| \le \frac{q}{2}n - n + 2$. If $n \ge 3$, then for all $i \ne 1$, $|G(i) \cap S'| \le \left(\frac{q-2}{2}\right)n + (n-1)$. Therefore $|S'| \le \frac{qmn}{2} - m - n + 4$, which is a contradiction.

Let $n = 2$ and $m = 2$. Then $|G(2) \cap S'| \le q - 2$. Therefore $|S'| \le q + q - 2 = 2q - 2 < 2q = \frac{qmn}{2} - m - n + 4$, which is a contradiction.

Let $n = 2$ and $m = 1$. Then $|S'| \le q < q + 1 = \frac{qmn}{2} - m - n + 4$, which is a contradiction. Therefore $p = i$ and $r = \frac{q+2}{2}$. In this case,
$|G(i) \cap S'| \le 3 + \left( \frac{q}{2}-1 \right)n,$
whereas
$|G(j) \cap S'| \le \left( \frac{q}{2}-1 \right)n + (n-1)$
for every $j \ne i$. Hence
$|S'| \le \frac{q}{2}mn - m - n + 4,$
again a contradiction. Therefore,
\[
\max\Bigl\{ |S| : S \subseteq V(G),\ |E(G[S])| = 1 \Bigr\} = \frac{q}{2}mn - m - n + 4.
\]
By the characterization of the $\vv$-number,
\[
\begin{aligned}
\vv\!\left(J_{m,n}^{(q)}\right) &= qmn - \left( \frac{q}{2}mn - m - n + 4 \right)
= \frac{q}{2}mn + m + n - 4.
\end{aligned}
\]
\end{proof}

The preceding formulas reveal two distinct asymptotic scales for the symbolic powers of $J_{m,n}$. On the one hand,
\[
\lim_{q \to \infty} \frac{\alpha(J_{m,n}^{(q)})}{q} = \lim_{q \to \infty} \frac{\vv(J_{m,n}^{(q)})}{q} = \frac{mn}{2} = \widehat{\alpha}(J_{m,n}),
\]
whereas
\[
\lim_{q \to \infty} \frac{\omega(J_{m,n}^{(q)})}{q} = \lim_{q \to \infty} \frac{\reg(R/J_{m,n}^{(q)})}{q} = m(n-1).
\]
Moreover,
\[
\omega(J_{m,n}^{(q)}) - \alpha(J_{m,n}^{(q)}) = \left\lfloor \frac{q}{2} \right\rfloor m(n-2),
\]
while
\[
\vv(J_{m,n}^{(q)}) - \alpha(J_{m,n}^{(q)}) =
\begin{cases}
-1, & \text{if } q \text{ is odd and } m \ne n,\\
0, & \text{if } q \text{ is odd and } m = n,\\
m+n-4, & \text{if } q \text{ is even}.
\end{cases}
\]
Thus, the $\vv$-number and the initial degree have the same asymptotic slope, namely $\frac{mn}{2}$, whereas the maximal degree of a minimal generator and the regularity have the larger asymptotic slope $m(n-1)$.

The precise relationships among the initial degree, the $\vv$-number, the maximal degree of a minimal generator, and the regularity are summarized in Table~\ref{comparison}.

\begin{table}[ht]
\centering
\renewcommand{\arraystretch}{1.3}
\begin{tabular}{|l|l|}
\hline
\textbf{Case} & \textbf{Relation among the invariants} \\
\hline
$m<n$, $q=1$, $n \ge 2m-1$
&
$\alpha-1=\vv=\omega-1=\reg$
\\
\hline
$m<n$, $q=1$, $m \le n<2m-1$
&
$\alpha-1=\vv=\omega-1<\reg$
\\
\hline
$m<n$, $q \ge 3$ odd
&
$\alpha-1=\vv<\omega-1=\reg$
\\
\hline
$(m,n,q)=(2,3,2)$
&
$\alpha-1<\vv=\omega-1=\reg$
\\
\hline
$m<n$, $q \ge 4$ even
&
$\alpha-1<\vv<\omega-1=\reg$
\\
\hline
$m=n$, $q=1$, $\reg(J)=m(m-1)+1$
&
$\alpha-1=\omega-1<\vv=\reg$
\\
\hline
$m=n$, $q=1$, $\reg(J)>m(m-1)+1$
&
$\alpha-1=\omega-1<\vv<\reg$
\\
\hline
$m=n=2$, $q \ge 2$
&
$\alpha-1=\omega-1<\vv=\reg$
\\
\hline
$m=n>2$, $q \ge 2$, $(m,q)=(3,2)$
&
$\alpha-1<\vv=\omega-1<\reg$
\\
\hline
$m=n>2$, $q \ge 2$, otherwise
&
$\alpha-1<\vv<\omega-1<\reg$
\\
\hline
\end{tabular}

\caption{Comparison of the invariants
$\alpha(J_{m,n}^{(q)})-1$,
$\vv(J_{m,n}^{(q)})$,
$\omega(J_{m,n}^{(q)})-1$, and
$\reg(R/J_{m,n}^{(q)})$.}\label{comparison}
\end{table}

\begin{remark}
\begin{enumerate}
\item Saha asked in \cite[Question~3.12]{Sa24} whether there exists a graph $G$ that is not complete multipartite and satisfies
\[
\vv(J(G)) > \omega(J(G))-1.
\]
The existence of such graphs was established in \cite[Example~8.3]{KMT25} and independently in \cite[Corollary~4.3]{VP26}. The diagonal chessboard complexes provide another family of examples. Indeed, by Theorem~\ref{vnumber-ordinary},
\[
\vv(J_{m,m}) > \omega(J_{m,m})-1 \quad \text{for all } m \ge 2.
\]

\item The formulas for the regularity and the $\vv$-number show that their behavior along the symbolic powers can be markedly different. More precisely, it follows from Theorems~\ref{reg-sym} and~\ref{asyv} that for $m\ge 3,$
\[
\reg\!\left(R/J_{m,n}^{(q)}\right) - \vv\!\left(J_{m,n}^{(q)}\right) \longrightarrow \infty
\qquad \text{as } q \to \infty.
\]
In particular, the regularity and the $\vv$-number have different asymptotic growth rates along the symbolic powers.
\end{enumerate}
\end{remark}

The explicit formulas for the initial degree and the $\vv$-number also yield nonvanishing results for the graded Betti numbers. The following corollary gives a general lifting principle and, in particular, yields several nonvanishing Betti numbers of $R/J_{m,n}^{(q)}$.

\begin{corollary}\label{betti-nonvanishing}
Suppose that either $n \ge 2m-1$, or $m < n < 2m-1$ and $q$ is even. Set
$\delta_q = q(m+n-2).$
If
$\beta_{i,j}\!\left(R/J_{m-1,n-1}^{(q)}\right) \ne 0,$
then
\[
\beta_{i,j+\delta_q}\!\left(R/J_{m,n}^{(q)}\right) \ne 0
\quad \text{and} \quad
\beta_{i+1,j+\delta_q+1}\!\left(R/J_{m,n}^{(q)}\right) \ne 0.
\]
In particular, setting
$a_q = \alpha\!\left(J_{m-1,n-1}^{(q)}\right),$
$v_q = \vv\!\left(J_{m-1,n-1}^{(q)}\right),$
we have
\[
\beta_{1,a_q+\delta_q}\!\left(R/J_{m,n}^{(q)}\right) \ne 0,
\qquad
\beta_{2,a_q+\delta_q+1}\!\left(R/J_{m,n}^{(q)}\right) \ne 0,
\]
and
\[
\beta_{2,v_q+\delta_q+2}\!\left(R/J_{m,n}^{(q)}\right) \ne 0,
\qquad
\beta_{3,v_q+\delta_q+3}\!\left(R/J_{m,n}^{(q)}\right) \ne 0.
\]
\end{corollary}

\begin{proof}
By Theorem~\ref{main-classification}, $R_{m,n}^q$ is vertex decomposable under either of the stated hypotheses. Moreover, by the proof of that theorem, $x_{1,1,1}$ is a shedding vertex. Set
$R_1 = R_{m,n}^q \setminus \{x_{1,1,1}\},~$
$R_2 = R_{m,n}^q \setminus N_{R_{m,n}^q}[x_{1,1,1}].$
Then $R_2 \cong R_{m-1,n-1}^q$. Hence, by \cite[Theorem~3.1]{MorKho16},
\[
\beta_{i,j}\bigl(R/J(R_{m,n}^q)\bigr) = \beta_{i,j-1}\bigl(R/J(R_1)\bigr) + \beta_{i,j-\delta_q}\bigl(R/J(R_2)\bigr) + \beta_{i-1,j-\delta_q-1}\bigl(R/J(R_2)\bigr).
\]
Suppose that $\beta_{i,j}\!\left(R/J_{m-1,n-1}^{(q)}\right) \ne 0$. Since $R_2 \cong R_{m-1,n-1}^q$, we have $\beta_{i,j}\!\left(R/J(R_2)\right) \ne 0$. Taking $(i,j+\delta_q)$ in the above recursion, we obtain
\[
\begin{aligned}
\beta_{i,j+\delta_q}\bigl(R/J(R_{m,n}^q)\bigr) &= \beta_{i,j+\delta_q-1}\bigl(R/J(R_1)\bigr) 
+ \beta_{i,j}\bigl(R/J(R_2)\bigr) + \beta_{i-1,j-1}\bigl(R/J(R_2)\bigr).
\end{aligned}
\]
The second summand is nonzero, and hence $\beta_{i,j+\delta_q}\!\left(R/J_{m,n}^{(q)}\right) \ne 0$. Similarly, taking $(i+1,j+\delta_q+1)$ gives
\[
\begin{aligned}
\beta_{i+1,j+\delta_q+1}\bigl(R/J(R_{m,n}^q)\bigr) &= \beta_{i+1,j+\delta_q}\bigl(R/J(R_1)\bigr) 
+ \beta_{i+1,j+1}\bigl(R/J(R_2)\bigr) + \beta_{i,j}\bigl(R/J(R_2)\bigr).
\end{aligned}
\]
The third summand is nonzero. Therefore, $\beta_{i+1,j+\delta_q+1}\!\left(R/J_{m,n}^{(q)}\right) \ne 0$. This proves the general assertion.

For the particular statements, by the definition of $a_q$, $\beta_{1,a_q}\!\left(R/J_{m-1,n-1}^{(q)}\right) \ne 0$. Applying the general assertion with $(i,j) = (1,a_q)$ yields
\[
\beta_{1,a_q+\delta_q}\!\left(R/J_{m,n}^{(q)}\right) \ne 0
\quad \text{and} \quad
\beta_{2,a_q+\delta_q+1}\!\left(R/J_{m,n}^{(q)}\right) \ne 0.
\]
On the other hand, by \cite[Corollary~5.3]{KMT25} and the definition of $v_q$,
$\beta_{2,v_q+2}\!\left(R/J_{m-1,n-1}^{(q)}\right) \ne 0.$
Applying the general assertion with $(i,j) = (2,v_q+2)$ gives
\[
\beta_{2,v_q+\delta_q+2}\!\left(R/J_{m,n}^{(q)}\right) \ne 0
\quad \text{and} \quad
\beta_{3,v_q+\delta_q+3}\!\left(R/J_{m,n}^{(q)}\right) \ne 0.
\]
This completes the proof.
\end{proof}

We now determine the $\vv$-number of the ordinary powers. Its behavior is particularly simple: it is linear in the exponent, with the diagonal case $m=n$ exhibiting a distinct behavior from the off-diagonal case.

\begin{theorem}\label{vnumber-ordinary}
For every $q \ge 1$,
\[
\vv\!\left(J_{m,n}^q\right) =
\begin{cases}
qm(n-1)-1, & \text{if } m \ne n,\\[2mm]
qm(n-1), & \text{if } m = n.
\end{cases}
\]
\end{theorem}

\begin{proof}
Since $J_{m,n} = J(R_{m,n})$, and $R_{m,n}$ is the line graph of the complete bipartite graph $K_{m,n}$, the graph $R_{m,n}$ is perfect. Consequently, by \cite[Theorem~2.10]{Villa08}, the Rees algebra
$\mathcal{R}(J(R_{m,n})) = \bigoplus_{q \ge 0} J(R_{m,n})^q t^q$
is normal. Hence, by \cite[Corollary~1.6]{HQ15}, $J(R_{m,n})$ satisfies the strong persistence property.

Suppose first that $m \ne n$. By Theorem~\ref{asyv},
$\vv(J(R_{m,n})) = m(n-1)-1,$
while $\omega(J(R_{m,n})) = m(n-1)$. Therefore, \cite[Proposition~3.4]{BMS24} yields
$\vv(J(R_{m,n})^q) = qm(n-1)-1$
for every $q \ge 1$.

Now suppose that $m=n$. By \cite[Remark~5.4]{Fi25},
$\vv(J(R_{m,m})^{q+1}) \le \vv(J(R_{m,m})^q) + \omega(J(R_{m,m})).$
By Theorem~\ref{asyv},
$\omega(J(R_{m,m})) = \vv(J(R_{m,m})) = m(m-1).$
We first prove that $\vv(J(R_{m,m})^q) \le qm(m-1)$ for all $q \ge 1$. The assertion is clear for $q=1$. If it holds for some $q \ge 1$, then
\[
\begin{aligned}
\vv(J(R_{m,m})^{q+1}) &\le \vv(J(R_{m,m})^q) + \omega(J(R_{m,m})) \\
&\le qm(m-1) + m(m-1) 
= (q+1)m(m-1).
\end{aligned}
\]
Thus, $\vv(J(R_{m,m})^q) \le qm(m-1)$ for every $q \ge 1$.

It remains to prove the reverse inequality. By symmetry, without loss of generality, let
\[
\mathfrak{p} = (x_{1,1}, \ldots, x_{1,p}) \in \operatorname{Ass}(R/J(R_{m,m})^q), ~ p \ge 2,
\]
and let $u$ be a monomial such that $(J(R_{m,m})^q : u) = \mathfrak{p}$. Since $x_{1,i} \in \mathfrak{p}$ for $1 \le i \le p$, we have $x_{1,i}u \in J(R_{m,m})^q$. Hence, for each $i$, there exists a minimal generator $m_i \in \mathcal{G}(J(R_{m,m})^q)$ such that
$\frac{m_i}{\gcd(m_i,u)} = x_{1,i}.$
Every minimal generator of $J(R_{m,m})^q$ has degree $qm(m-1)$. Consequently, $\deg\gcd(m_i,u) = qm(m-1)-1$, and hence $\deg(u) \ge qm(m-1)-1$.

Suppose, to the contrary, that $\deg(u) = qm(m-1)-1$. Then $x_{1,i}u \in \mathcal{G}(J(R_{m,m})^q)$ for $1 \le i \le p$. In particular, write
$x_{1,1}u = m_1' m_2' \cdots m_q',$
where $m_1', \ldots, m_q' \in \mathcal{G}(J(R_{m,m}))$. After relabeling the factors, we may assume that $x_{1,1} \mid m_1'$. Write $m_1' = x_{1,1} n_1'$. 
Let
$S_1 = \{x_{1,1}, x_{2,1}, \ldots, x_{m,1}\}.$
Every minimal generator of $J(R_{m,m})$ contains exactly $m-1$ vertices of $S_1$. Since $x_{1,1} \mid m_1'$, there is a unique vertex $y \in S_1 \setminus \{x_{1,1}\}$ which does not divide $m_1'$. In particular, $y \nmid n_1'$.
Now,
$x_{1,2}u = \widetilde{m}_1 m_2' \cdots m_q' \in \mathcal{G}\!\left(J(R_{m,m})^q\right)$
where $\widetilde{m}_1 = x_{1,2} n_1'$. However, $\widetilde{m}_1$ contains only $m-2$ vertices of $S_1$, whereas each of $m_2', \ldots, m_q'$ contains exactly $m-1$ vertices of $S_1$. Hence
\[
\sum_{x \in S_1} \deg_x(x_{1,2}u) = (m-2) + (q-1)(m-1) = q(m-1)-1.
\]
This contradicts the fact that every minimal generator of $J(R_{m,m})^q$ contains exactly $q(m-1)$ vertices of $S_1$, counted with multiplicity. Thus $\deg(u) \ne qm(m-1)-1$. Together with $\deg(u) \ge qm(m-1)-1$, this gives $\deg(u) \ge qm(m-1)$. Therefore,
$\vv(J(R_{m,m})^q) \ge qm(m-1).$
Combining this with the upper bound proves $\vv(J(R_{m,m})^q) = qm(m-1)$.
\end{proof}

\begin{remark}
For the family of cover ideals \(J_{m,n}\), the following two observations follow from our results.
\begin{enumerate}
\item In \cite[Question~5.3(b)]{FS24}, the following question was posed: is it true that
\[
\vv(I^k) < \vv(I^{k+1}) \quad \text{for all } k \ge 1?
\]
For the ideals \(J_{m,n}\), the answer is affirmative. Indeed, by Theorem~\ref{vnumber-ordinary},
\[
\vv(J_{m,n}^{q+1}) - \vv(J_{m,n}^q) = m(n-1) > 0 \quad \text{for all } q \ge 1.
\]
Thus,
$\vv(J_{m,n}^q) < \vv(J_{m,n}^{q+1}) ~ \text{for all } q \ge 1.$

\item For every fixed \(q \ge 2\), the difference
$\vv\!\left(J_{m,n}^q\right) - \vv\!\left(J_{m,n}^{(q)}\right)$
is unbounded as \(m,n\) increase. This follows from the explicit formulas in Theorems~\ref{asyv} and~\ref{vnumber-ordinary}.
\end{enumerate}
\end{remark}

The explicit formulas for the $\vv$-number, together with the regularity formulas above, yield the following comparison.

\begin{corollary}\label{vnumber-reg}
If $m \ne n$, then
$\vv\!\left(J_{m,n}^q\right) \le \reg\!\left(R/J_{m,n}^q\right)$
for every $q \ge 1$. If $m=n$, then
$\vv\!\left(J_{m,m}^2\right) \le \reg\!\left(R/J_{m,m}^2\right).$
\end{corollary}

\begin{proof}
The first assertion follows immediately from Theorem~\ref{vnumber-ordinary} and the formula for $\reg(R/J_{m,n}^q)$. The second assertion follows from Theorem~\ref{equal-not}.
\end{proof}

The following examples show that the inequality in Corollary~\ref{vnumber-reg} can be strict.

\begin{example}
For $(m,n) = (4,5)$, we have
$\vv\!\left(J_{4,5}^2\right) = 31 < 32 = \reg\!\left(R/J_{4,5}^2\right).$
Similarly, for $m=n=4$,
$\vv\!\left(J_{4,4}^2\right) = 24 < 25 = \reg\!\left(R/J_{4,4}^2\right).$
Here, the regularity values are computed using \textsc{Macaulay2} \cite{M2} over a field of characteristic zero, while the values of the $\vv$-numbers follow from Theorem~\ref{vnumber-ordinary}.
\end{example}

\section{Weakly Polymatroidal Ideals}\label{wp-sec}

In this section, we investigate the weak polymatroidality of
$J_{m,n}$. Lu and Wang \cite{LW24} conjectured that
$I_{\Delta^\vee}$ is weakly polymatroidal whenever $\Delta$ is a pure
vertex decomposable simplicial complex; see Conjecture~\ref{lw-conj}.
Since, by \cite{Z94}, $\Delta_{m,n}$ is pure and vertex decomposable if
and only if $n\ge2m-1$, chessboard complexes provide a natural family
for testing this conjecture. We show that the conjecture fails for this
family.

\begin{theorem}\label{conter-ex}
For every $m\ge3$, the ideal $J_{m,2m-1}$ is not weakly
polymatroidal.
\end{theorem}

\begin{proof}
Suppose, to the contrary, that $J_{m,2m-1}$ is weakly polymatroidal.
Then there exists a total ordering $<$ on the variables such that
$\mathcal{G}(J_{m,2m-1})$ satisfies the weak exchange property with
respect to $<$. 
For each row of $\Delta_{m,2m-1}$, let $S$ denote the set of the first $m-1$ variables from each row with respect to $<$. Let
$x_{u,v}
=
\min\bigl(V(\Delta_{m,2m-1})\setminus S\bigr).$
By symmetry, we may assume that $x_{u,v}$ belongs to the first row.
Write the variables in the first row as
\[
x_{1,p_1}
<
x_{1,p_2}
<
\cdots
<
x_{1,p_{m-1}}
<
x_{1,r_1}
<
x_{1,r_2}
<
\cdots
<
x_{1,r_m},
\]
so that
$x_{u,v}=x_{1,r_1}.$
Set
$S'
=
\{x_{i,r_j}\mid 2\le i\le m,\ 2\le j\le m\}.$

We distinguish two cases.

\medskip

\noindent\textbf{Case 1.}
Suppose that
$S'\cap
\bigl(V(\Delta_{m,2m-1})\setminus S\bigr)
\ne\emptyset.$
By symmetry, we may assume that
$x_{2,r_2}\in
S'\cap
\bigl(V(\Delta_{m,2m-1})\setminus S\bigr).$
Consider the monomials
\[
u=
\prod_{t\ne r_2}x_{1,t}
\prod_{t\ne r_1}x_{2,t}
\prod_{t\ne r_3}x_{3,t}
\cdots
\prod_{t\ne r_m}x_{m,t}
\text{ and }
v=
\prod_{t\ne r_1}x_{1,t}
\prod_{t\ne r_2}x_{2,t}
\prod_{t\ne r_3}x_{3,t}
\cdots
\prod_{t\ne r_m}x_{m,t}.
\]
Then $
u,v\in\mathcal{G}(J_{m,2m-1}),
$ with $
\deg_{x_{1,r_1}}(u)=1
~ \text{and} ~
\deg_{x_{1,r_1}}(v)=0.
$
However, the pair $(u,v)$ violates the weak exchange property with respect to the chosen ordering. This contradicts the weak polymatroidality of
$J_{m,2m-1}$.

\medskip

\noindent\textbf{Case 2.}
Suppose that
$S'\cap
\bigl(V(\Delta_{m,2m-1})\setminus S\bigr)
=
\emptyset.$
Let
$x_{e,f}
=
\min\bigl(
V(\Delta_{m,2m-1})\setminus(R_1\cup S)
\bigr),$
where $R_1$ denotes the first row. By symmetry, we may assume that
$x_{e,f}=x_{2,p_j}.$
Relabel the last $m$ variables in the second row as
$x_{2,q_1}
<
x_{2,q_2}
<
\cdots
<
x_{2,q_m},
$
where $q_1=p_j$. Since $S$ consists of the first $m-1$ variables in
each row, we have
$x_{2,q_i}\notin S~(1\le i\le m).$
In particular,
$x_{3,q_2}\notin S$
and
$x_{2,q_1}<x_{3,q_2}.$
Now consider
\[
u=
\prod_{t\ne q_3}x_{1,t}
\prod_{t\ne q_2}x_{2,t}
\prod_{t\ne q_1}x_{3,t}
\prod_{t\ne q_4}x_{4,t}
\cdots
\prod_{t\ne q_m}x_{m,t}
\]
and
\[
v=
\prod_{t\ne q_3}x_{1,t}
\prod_{t\ne q_1}x_{2,t}
\prod_{t\ne q_2}x_{3,t}
\prod_{t\ne q_4}x_{4,t}
\cdots
\prod_{t\ne q_m}x_{m,t}.
\]
Again,
$u,v\in\mathcal{G}(J_{m,2m-1}),$
but the pair $(u,v)$ violates the weak exchange property with respect
to the chosen ordering. This contradicts the weak polymatroidality of
$J_{m,2m-1}$.

Thus, in either case, we obtain a contradiction. Hence
$J_{m,2m-1}$ is not weakly polymatroidal.
\end{proof}

\vspace*{1mm}
\noindent
\textbf{Acknowledgments.}
We thank Seyed Fakhari and Amir Mafi for their helpful discussions,
insightful comments, and clarifications. The second author was partially
supported by the National Board for Higher Mathematics (NBHM) and the
Science and Engineering Research Board (SERB).

\vspace*{1mm}
\noindent
\textbf{Data availability statement.}
Data sharing is not applicable to this article, as no datasets were
generated or analyzed during the current study.

\vspace*{1mm}
\noindent
\textbf{Conflict of interest.}
The authors declare that they have no competing interests or personal
relationships that could have influenced the work reported in this
paper.

\bibliographystyle{abbrv}
\bibliography{refs_reg} 
\end{document}